\documentclass{gtart}

\usepackage[inner=25mm, outer=25mm, textheight=225mm]{geometry}

\usepackage{psfrag, graphicx, subfigure, epsfig,color}

\usepackage{amsmath, amssymb, latexsym, euscript}

\usepackage{mathptmx, wrapfig, setspace}

\usepackage[super]{nth}

\usepackage[margin=1cm]{caption}

\usepackage{tikz}
\usetikzlibrary{matrix}

\usepackage{titlesec}
\titleformat{\subsubsection}[runin]
       {\normalfont\bfseries}
       {\thesubsubsection}
       {0.5em}
       {}
\titlespacing*{\subsubsection}{0pt}{0pt}{0.5em}

\makeatletter
\newtheorem*{rep@theorem}{\rep@title}
\newcommand{\newreptheorem}[2]{%
\newenvironment{rep#1}[1]{%
 \def\rep@title{#2 \ref{##1}}%
 \begin{rep@theorem}}%
 {\end{rep@theorem}}}
\makeatother

\def\bkR{{\rm I\kern-.17em R}}

\def\bkZ{{\rm Z\kern-.28em Z}}
\def\Z{\bkZ}

\DeclareMathOperator{\Sym}{Sym}
\DeclareMathOperator{\Span}{span}
\DeclareMathOperator{\Int}{int}

\DeclareMathOperator{\lk}{link}

\theoremstyle{plain}
\newtheorem{theorem}{Theorem}
\newtheorem{lemma}[theorem]{Lemma}
\newtheorem{proposition}[theorem]{Proposition}
\newtheorem{corollary}[theorem]{Corollary}

\newtheorem{scholion}[theorem]{Scholion}
\newreptheorem{theorem}{Theorem}
\newreptheorem{proposition}{Proposition}

\theoremstyle{definition}
\newtheorem{definition}[theorem]{Definition}
\newtheorem*{definition*}{Definition}

\newtheorem{question}[theorem]{Question}

\numberwithin{equation}{section}

\begin{document}

\title{Multisections of piecewise linear manifolds}
\author{J.\thinspace Hyam Rubinstein and Stephan Tillmann}

\begin{abstract}
Recently Gay and Kirby described a new decomposition of smooth closed $4$--manifolds called a trisection. This paper generalises Heegaard splittings of 3-manifolds and trisections of 4-manifolds to all dimensions, using triangulations as a key tool. In particular, we prove that every closed piecewise linear $n$--manifold has a \emph{multisection}, i.e.\thinspace can be divided into $k+1$ $n$--dimensional 1--handlebodies, where $n=2k+1$ or $n=2k$, such that intersections of the handlebodies have spines of small dimensions. 
Several applications, constructions and generalisations of our approach are given.
\end{abstract}

\primaryclass{52B70, 53C21, 57N15}
\keywords{piecewise linear manifold, combinatorial manifold, triangulation, handlebody, CAT(0), cubing}
\maketitle


\section{Introduction}

In \cite{GK}, David Gay and Rob Kirby introduced a beautiful decomposition of an arbitrary smooth, oriented closed $4$--manifold, called \emph{trisection}, into three handlebodies glued along their boundaries as follows. Each handlebody is a boundary connected sum of copies of $S^1 \times B^3,$ and has boundary a connected sum of copies of $S^1 \times S^2.$ The triple intersection of the handlebodies is a closed orientable surface $\Sigma,$ which divides each of their boundaries into two $3$--dimensional handlebodies (and hence is a Heegaard surface). These  $3$--dimensional handlebodies are precisely the intersections of pairs of the 4--dimensional handlebodies.

In dimensions $\le 4,$ there is a bijective correspondence between isotopy classes of smooth and piecewise linear structures~\cite{C1, C2}, but this breaks down in higher dimensions. This paper generalises Gay and Kirby's concept of a trisection to higher dimensions in the piecewise linear category, and hence all manifolds, maps and triangulations are assumed to be piecewise linear unless stated otherwise. Our definition and results apply to any compact smooth manifold by passing to its unique piecewise linear structure \cite{W}.

The definition of a multisection, which generalises both that of a Heegaard splitting of a 3--manifold and that of a trisection of a 4--manifold, focuses on properties of spines. Let $N$ be a compact manifold with non-empty boundary. The subpolyhedron $P$ is a \emph{spine} of $N$ if $P \subset \Int(N)$ and $N$ PL collapses onto $P.$

\begin{definition}[(Multisection of closed manifold)]\label{def:multisection}
Let $M$ be a closed, connected, piecewise linear $n$--manifold. A \emph{multisection} of $M$ is a collection of $k+1$ piecewise linear submanifolds $H_i \subset M,$ where $0 \le i \le k$ and $n=2k$ or $n=2k+1,$ subject to the following four conditions:
\begin{enumerate}
\item Each $H_i$ has a single $0$--handle and a finite number, $g_i$, of $1$--handles, and is homeomorphic to a standard piecewise linear $n$--dimensional 1--handlebody of genus $g_i.$ 
\item The handlebodies $H_i$ have pairwise disjoint interior, and $M = \bigcup_i H_i.$
\item  The intersection $H_{i_1} \cap H_{i_2} \cap \ldots \cap H_{i_r}$ of any proper subcollection of the handlebodies is a compact, connected submanifold with boundary and of dimension $n-r+1.$ Moreover, it has a spine of dimension $r,$ except if $n=2k$ and $r=k,$ then there is a spine of dimension $r-1.$ Each such submanifold is called a \emph{multisection submanifold}.
\item The intersection $H_{0} \cap H_{1} \cap \ldots \cap H_{k}$ of all handlebodies is a closed, connected submanifold of $M^n$ of dimension $n-k,$ and called the \emph{central submanifold}. 
\end{enumerate}
\end{definition}
It follows from our definitions that the first condition in the above definition is equivalent to
\begin{itemize}
\item[($1'$)] Each $H_i$ has spine a graph with Euler characteristic $1-g_i.$
\end{itemize}
Proposition~\ref{pro:connectivity} implies that the connectivity requirement of the multisection submanifolds in (3) and the central submanifold in (4) follows from the condition on the codimensions of the spines in (3).

In the remainder of this introduction, we discuss some properties of multisections as well as directions for further research. It is our hope that this new structure and its derived invariants will help gain insights into new infinite families of piecewise linear manifolds, if not the realm of all such manifolds. 

\medskip

\textbf{Example.} (The tropical picture of complex projective space.)
Consider the map $\mathbb{C}P^n \to \Delta^n$ defined by
$$[\;z_0\; : \;\ldots\; :\; z_n\;] \;\mapsto\; \frac{1}{\sum|z_k|}\;(\;|z_0|\;,\; \ldots\;,\; |z_n|\;).$$
The \emph{dual spine} $\Pi^n$ in  $\Delta^n$ is the subcomplex of the first barycentric subdivision of $\Delta^n$ spanned by the 0--skeleton of the first barycentric subdivision 
minus the 0--skeleton of $\Delta^n.$ 
This is shown for $n=2$ in Figure\ref{fig:Pi2} and $n=3$ in Figure\ref{fig:Pi3}.
Decomposing along $\Pi^n$ gives $\Delta^n$ a natural
 \emph{cubical structure} with $n+1$ $n$--cubes, and the lower-dimensional cubes that we will focus on are the intersections of non-empty collections of these top-dimensional cubes. Each $n$--cube pulls back to a $2n$--ball in $\mathbb{C}P^n,$ and the collection of these balls is a multisection. For example, if $n=2,$ the 2--cubes pull back to 4--balls, each 1--cube pulls back to $S^1 \times D^2$ and the 0--cube pulls back to $S^1 \times S^1$ as shown in Figure\ref{fig:CP2}.

\begin{figure}[h]
  \centering
 \subfigure[$\Pi^2\subset \Delta^2$]{\includegraphics[height=3.4cm]{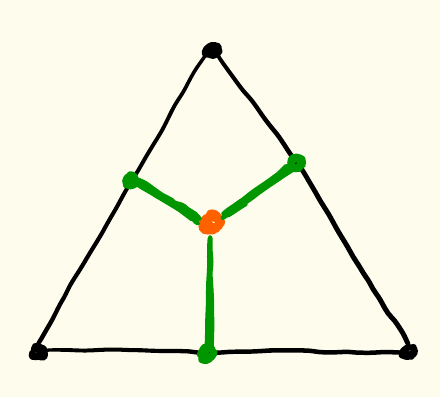}\label{fig:Pi2}}
    \qquad
     \subfigure[$\Pi^3\subset \Delta^3$]{\includegraphics[height=3.4cm]{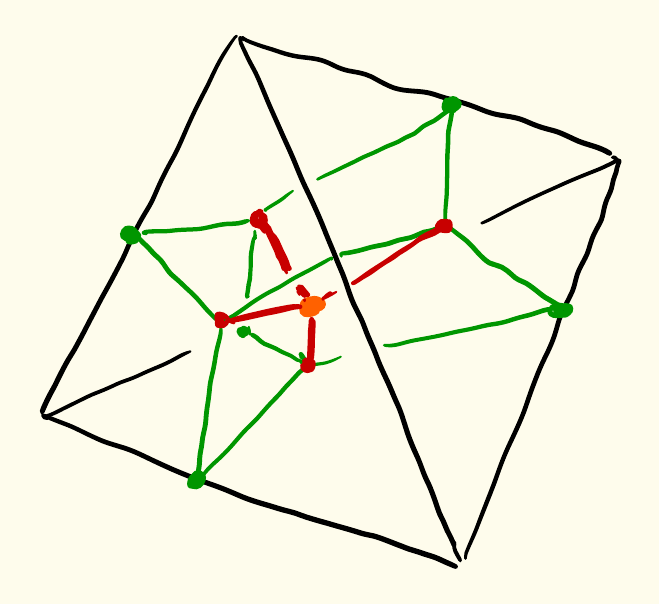}\label{fig:Pi3}}
     \qquad
      \subfigure[$\mathbb{C} P^2 \to \Delta^2$]{\includegraphics[height=3.4cm]{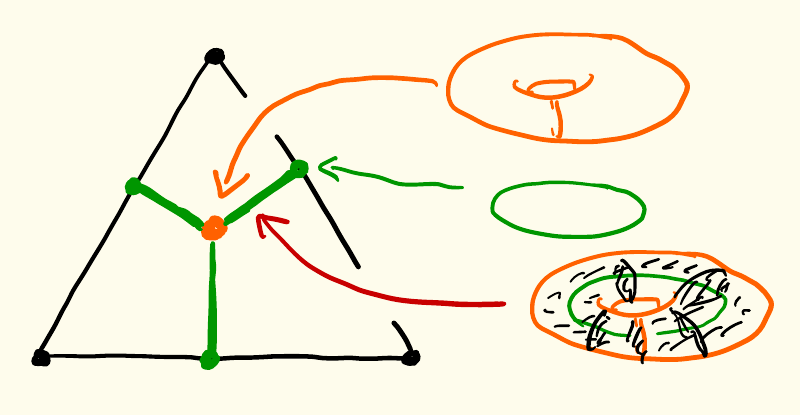}\label{fig:CP2}}
\caption{}
\end{figure}

\medskip

\textbf{Existence (\S\ref{sec:existence}).}
A multisection of a 1--manifold is just the 1--manifold.
The study of multisections in dimension 2 is the study of separating, simple, closed curves. 
A multisection of a 3--manifold is a Heegaard splitting (\S\ref{sec:construction dim 3}). A trisection in the sense of Gay and Kirby~\cite{GK} is a multisection of an orientable 4--manifold with the additional property that the handlebodies $H_j$ have the same genus. After showing that every multisection of an orientable 4--manifold can be modified to a trisection in the sense of \cite{GK} (\S\ref{sec:GKvsRT}), we will start to use the term trisection to apply to all multisections in dimension four, so that we can talk about \emph{bisections} ($n=2,3$), \emph{trisections} ($n=4,5$), \emph{quadrisections} ($n=6,7$), etc.\thinspace without further qualification.

\begin{theorem}
Every closed piecewise linear manifold has a multisection.
\end{theorem}

\textbf{Sketch of proof.} Suppose $M$ is a closed, connected, piecewise linear manifold of dimension $n.$ Our strategy is to construct a piecewise linear map $\phi\co M\to \sigma,$ where $\sigma$ is a $k$--simplex for $k$ satisfying $n=2k$ or $n=2k+1,$ and to obtain the multisection as the pull back of the cubical structure of $\sigma$ to $M.$ Our map $\phi$ will have the property that each vertex of $\sigma$ pulls back to a connected graph, and each top-dimensional cube pulls back to a regular neighbourhood of this graph, a 1--handlebody.

\begin{figure}[h]
  \centering
	{\includegraphics[height=3cm]{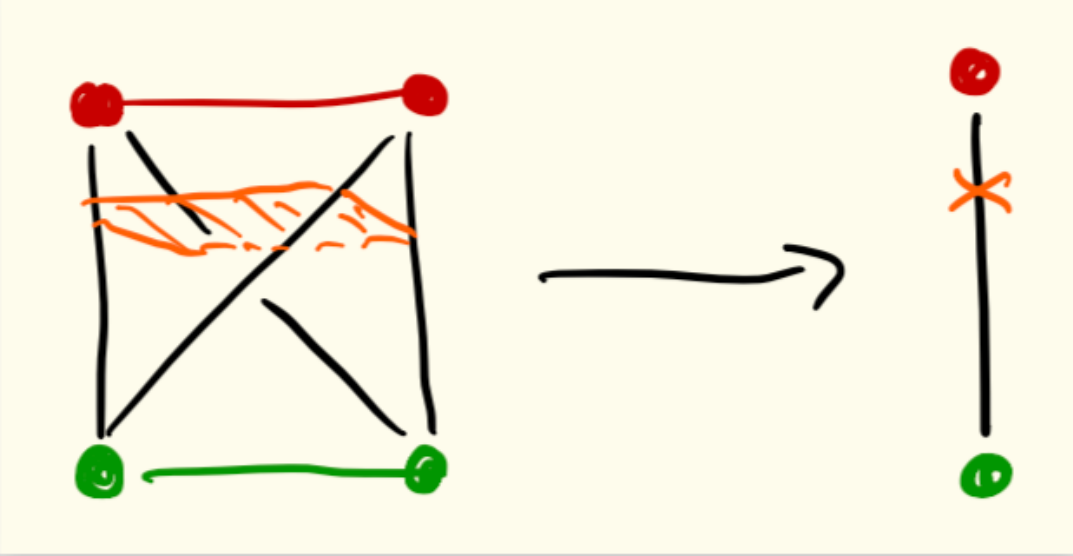}}
\caption{Partition map for $n=3$}\label{fig:2-d-partition}
\end{figure}

We use triangulations to define $\phi.$ Since a piecewise linear manifold admits a piecewise linear triangulation $|K|\to M$ (where the link of each simplex in the simplicial complex $K$ is equivalent to a standard piecewise linear sphere) we can and will assume that such a triangulation of $M$ is fixed. Since $M$ is closed, there is a finite number of simplices in the triangulation, and $\phi$ is uniquely determined by a partition of the vertices of the triangulation into $k+1$ sets, and a bijection between the sets in this partition and the vertices of $\sigma.$ We call such a map $M\to \sigma$ a \emph{partition map}. To ensure that the cubical structure of $\sigma$ pulls back to submanifolds with the required properties, we determine suitable combinatorial properties on the triangulation. In the odd-dimensional case, we show that the first barycentric subdivision of any triangulation has a suitable partition (see \S\ref{sec:construction dim 3} and \S\ref{sec:construction dim odd}). Moreover the $r$--dimensional spine of the intersection of $r$ handlebodies meets each top-dimensional simplex in $M$ in exactly one $r$--cube. In even dimensions, we obtain an analogous result after performing bistellar moves on this subdivision (see  \S\ref{sec:construction dim 4} and \S\ref{sec:construction dim even}). \qed

We say that a triangulation \emph{supports a multisection} if there is a partition of the vertices defining a partition map $\phi\co M\to \sigma$ with the property that the pull back of the cubical structure is a multisection. Special properties of triangulations may imply special properties of the supported multisections and vice versa. For instance, special properties of a Heegaard splitting of a 3--manifold are shown in \cite{HRST2015} to imply special properties of the dual triangulation. The cornerstone of the modern development of Heegaard splittings is the work of Casson and Gordon~\cite{CG}, and it is a tantalising problem to generalise this to higher dimensions.

\medskip

We next discuss a number of general structure results.

\medskip

\textbf{Recursive structure and generalisations (\S\ref{sec:recursive}).}
A multisection of $M$ induces a stratification of the boundary of each handlebody into lower dimensional manifolds.
We will refer to this as the \emph{recursive structure} of a multisection. For example, for a trisection ($n=4, 5$) the boundaries of each handlebody are divided into two submanifolds. For a quadrisection ($n=6, 7$) the boundaries of the handlebodies are divided into three submanifolds with connected pairwise intersections. These decompositions satisfy part of the definition of a multisection, namely the submanifolds have spines of small dimensions, but the top dimensional submanifolds are not necessarily handlebodies. From the viewpoint of Martelli's complexity theory~\cite{Mar2010}, such generalised multisections may be a fruitful approach to the study of classes of examples. For instance, a decomposition of a 4--manifold into 4--dimensional 1-- or 2--handlebodies is a decomposition into 4--manifolds of complexity 0.

\medskip

\textbf{Non-positively curved cubings from multisections (\S\ref{sec:CAT(0)}).}
The partition map $\phi\co M \to \sigma$ can be used to pull back the cubical structure of the target simplex. This gives a natural cell decomposition of the submanifolds in a multisection, with cells of very simple combinatorial types. In the case of the closed, central submanifold $\Sigma=H_0 \cap H_1 \cap \dots \cap H_{k}$ this is a cubing. We show that a triangulation and a partition map can be chosen such that the cubing of $\Sigma$ satisfies the Gromov link conditions \cite{Gr}, and hence is non-positively curved:
\begin{theorem}\label{thm:CAT(0)}
Every piecewise linear manifold has a triangulation supporting a multisection, such that the central submanifold has a non-positively curved cubing.
\end{theorem}

Since a $(2k+1)$--manifold has a central submanifold of dimension $k+1,$ this result produces manifolds with non-positively curved cubings in each dimension. We also remark that our construction yields cubings with precisely one top-dimensional cube in the central submanifold for each top-dimensional simplex in the triangulation of the manifold.

\begin{question}
What conditions does a non-positively curved cubed $k$--manifold need to satisfy so that it is PL--homeomorphic to the central submanifold in a multisection of a $(2k+1)$--manifold or a $2k$--manifold?
\end{question}

\medskip

\textbf{Structure of fundamental group (\S\ref{sec:fundamental group}) and higher homotopy groups (\S\ref{sec:Recursive structure and higher homotopy groups}).}
A multisection of a manifold gives a decomposition of its fundamental group as a \emph{generalised} graph of groups (which is defined exactly as a graph of groups with the only modification that the homomorphisms from edge groups to vertex groups need not be monomorphisms; cf.\thinspace \cite[\S6.2]{Geo}). In particular, since every finitely presented group is the fundamental group of a closed 4--manifold, we obtain the following decomposition theorem, which may be of independent interest.

\begin{proposition}\label{pro:structure of fund gp}
Every finitely presented group has a generalised graph of groups decomposition as shown in Figure~\ref{fig:structure of groups_intro}, where the vertex groups $\Gamma_0, \Gamma_1, \Gamma_2$ are free of rank $\le g,$ $\Gamma_g$ is the fundamental group of a closed orientable surface of genus $g,$ all edge groups are naturally isomorphic with $\Gamma_g$ and the oriented edges represent epimorphisms.
\end{proposition}

\begin{figure}[h]
\begin{center}
\begin{tikzpicture}
  \matrix (m) [matrix of math nodes, row sep=1em, column sep=1em]{
    &\Gamma_0& \\
    &\Gamma_g&\\
    \Gamma_1&&\Gamma_{2}& \\};
  \path[-stealth]
  (m-2-2) edge (m-1-2)
              edge (m-3-1)
              edge (m-3-3);
\end{tikzpicture}
\end{center}
\caption{Structure of finitely presented groups}
\label{fig:structure of groups_intro}
\end{figure}
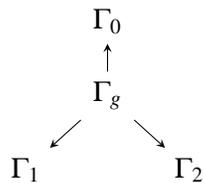

Abrams, Gay and Kirby~\cite{AGK} analysed this structure of the fundamental group in more detail to define \emph{group trisections} and formulate a group theoretic statement equivalent to the smooth Poincar\'e conjecture.

We observe that the recursive structure of a multisection enables both the fundamental group (\S\ref{sec:fundamental group}) and higher homotopy groups (\S\ref{sec:Recursive structure and higher homotopy groups}) to be computed using the multisection submanifolds and the central submanifold. As a corollary, we show that multisections lift to covering spaces. In particular, in higher dimensions the role of the surface group can be played by the fundamental group of a non-positively curved cube complex.

\medskip

\textbf{Invariants.}
The Heegaard genus of a 3--manifold has natural generalisations to multisections. For instance,
each of $S^4$ and $\mathbb{C} P^2$ have trisections with three 4--balls, but they are distinguished by the genera of the pairwise intersections. The \emph{genus of a trisection} of a 4--manifold, say $M = H_0 \cup H_1 \cup H_2,$ can be defined as the minimum (with respect to the lexicographic ordering) of all 4--tuples of the form
$$(\;g(H_i), g(H_j), g(H_k), g(H_i\cap H_j\cap H_k)\;).$$
The \emph{multisection genus} of $M$ is then the minimum genus of any trisection of $M.$ In higher dimensions, genera need to be replaced by other invariants. These invariants are also related to the associated generalised graph of groups decomposition of the fundamental group.

\begin{question}
Are there interesting families of 4--manifolds for which there exists an algorithm to compute a multisection of minimal genus for each member of the family?
\end{question}

\medskip

\textbf{Uniqueness.}
There is a natural stabilisation procedure of multisections. In dimension 3, this increases the genus of both 3--dimensional handlebodies, whilst in higher dimensions, this increases the genus of just one of the top-dimensional handlebodies. The Reidemeister-Singer theorem~\cite{Re, Si}  states that any two Heegaard splittings of a 3--manifold have a common stabilisation.
Using in an essential way the uniqueness up to isotopy of genus $g$ Heegaard splittings of $\#^k (S^2\times S^1)$ due to Waldhausen~\cite{Wald1968}, Gay and Kirby~\cite{GK} show that any two trisections of a 4--manifold have a common stabilisation up to isotopy. Combined with  \S\ref{sec:GKvsRT}, this implies that any two multisections of a 4--manifold also have a common stabilisation up to isotopy. We do not give an independent proof of this fact in this paper.

\begin{question}
Under what conditions is there a common stabilisation for two given multisections of a manifold of dimension at least five? 
\end{question}

Our existence proof constructs multisections dual to triangulations. Conversely, up to possibly stabilising the multisection, one can build triangulations dual to multisections. However, stabilisation in higher dimensions adds summands of $S^1 \times S^{n-k-1}$ to the central submanifold, and hence we expect the equivalence relation generated by stabilisation to be finer than PL equivalence of the dual triangulations since central submanifolds with a negatively curved cubing are aspherical.

\medskip

\textbf{Constructions and examples (\S\ref{sec:examples}).} 
An extended set of examples of trisections of 4--manifolds can be found in \cite{GK}. The recent work of Gay \cite{Gay2015}, Meier, Schirmer and Zupan~\cite{MSZ, MZ-standard, MZ-bridge} gives some applications and constructions arising from trisections of 4--manifolds and relates them to other structures. In this section, we outline some further constructions, focussing on arbitrary dimensions. 
The above existence result shows that 
first barycentric subdivision leads to triangulations supporting multisections. However, more efficient triangulations can be identified using the symmetry representations of \cite{RT} (see \S\ref{sec:constructing with symmetric representations}). We give a number of applications of this approach, including \emph{generalised multisections} and \emph{twisted multisections}. We also discuss \emph{connected sums}, a \emph{Dirichlet construction}, \emph{products} and the case of \emph{manifolds with non-empty boundary.}

\medskip

\textbf{Acknowledgements\ } 
The authors are partially supported under the Australian Research Council's Discovery funding scheme (project numbers DP130103694 and DP160104502). 
The authors would like to thank David Gay for suggesting to the first author that there might be an approach to trisections via triangulations. 


\section{Existence of multisections}
\label{sec:existence}

The outline of the existence proof is given in the introduction. We give a motivation for both the definition of a multisection and the strategy of the existence proof in dimensions three (\S\ref{sec:construction dim 3}) and four (\S\ref{sec:construction dim 4}), followed by the general arguments for arbitrary odd (\S\ref{sec:construction dim odd}) and even (\S\ref{sec:construction dim even}) dimensions. We also clarify the relationship between multisections of 4--manifolds and the trisections of Gay and Kirby (\S\ref{sec:GKvsRT}).

All manifolds, maps and triangulations are assumed to be piecewise linear (PL) unless stated otherwise (in which case we will emphasise this by saying topological manifold, \ldots). Our main reference on PL topology is Rourke and Sanderson~\cite{RS}. A primer can be found in Thurston~\cite[\S3.9]{T}, and a collection of basic definitions and tools in Martelli~\cite[\S2]{Mar2010}. 


\subsection{Closed {3}--manifolds have multisections}
\label{sec:construction dim 3}

We recall the classical existence proof of Heegaard splittings (see, for instance, \cite{He}), which motivates our definition in higher dimensions, and provides a model for the existence proofs.
Suppose that $M$ is a triangulated, closed, connected $3$--manifold, and
there is a partition $\{P_0, P_1\}$ of the set of all vertices in the triangulation, such that 
\begin{enumerate}
\item[$(1_3)$] for each set $P_k,$ every tetrahedron has a pair of vertices in the set; and
\item[$(2_3)$] the union of all edges with both ends in $P_k$ is a connected graph $\Gamma_k$ in $M.$
\end{enumerate}
We can form regular neighbourhoods of each of these graphs $\Gamma_k,$ which are handlebodies $H_0,$ $H_1$ respectively, such that the handlebodies meet along their common boundary $\Sigma,$ which is a normal surface consisting entirely of quadrilateral disks, one in each tetrahedron, separating the vertices in $P_0,$ $P_1$ (see Figure~\ref{fig:2-d-partition}). Hence $\Sigma$ is a Heegaard surface in $M.$ A triangulation with the desired properties is obtained as follows. Suppose $|K|\to M$ is a triangulation of $M,$ and take the first barycentric subdivision $K'$ of $K.$ 
Let $P_0$ be the set of all vertices of $K$ and barycentres of edges of $K;$ and let $P_1$ be the set of all barycentres of the triangles and the tetrahedra of $K.$ Then $\{P_0, P_1\}$ is a partition of the vertices of  $K'$ satisfying $(1_3)$ and $(2_3).$ Moreover, the vertices of the cubulated surface $\Sigma$ have degrees $4$ or $6,$ and hence $\Sigma$ is a non-positively curved cube complex.

Examples of triangulations of manifolds that satisfy $(1_3)$ and $(2_3),$ but are not barycentric subdivisions, are the standard 2--vertex triangulations of lens spaces. See \S\ref{sec:constructing with symmetric representations} for a strategy to identify triangulations dual to multisections.

The partition $\{P_0, P_1\}$ defines a piecewise linear map $\phi\co M \to [0,1]$ by $\phi(P_0)=0$ and $\phi(P_1)=1$. This is often called a \emph{height function} and we refer to it as a \emph{partition map}. The pre-image $\phi^{-1} (\frac{1}{2})$ is a Heegaard surface $\Sigma$ for $M$ as described above. 
The inverse image of any point in the interior of $[0,1]$ is a normal surface isotopic to $\Sigma$. The intersection of this inverse image with any tetrahedron of $\mathcal T$  is a quadrilateral disk ($2$--cube). The inverse image of either endpoint $0$ or $1$ is a graph and its intersection with any tetrahedron is an edge ($1$-cube). The division of the closed interval ($1$--simplex) into two half intervals is the dual decomposition into $1$--cubes. An analogous decomposition is exactly what we will use in higher dimensions. 


\subsection{Closed 4--manifolds have multisections}
\label{sec:construction dim 4}

Let $M$ be a closed, connected $4$--manifold with piecewise linear triangulation $|K|\to M.$ We assume that there is a partition $\{P_0, P_1, P_3\}$ of the set of all vertices of $K$ with the following properties:
\begin{enumerate}
\item[$(0'_4)$]  every 4--simplex meets each of the sets $P_0$ and $P_1$ in two vertices and $P_2$ in a single vertex; and
\item[$(0''_4)$]  the graph $\Gamma_k$ consisting of all edges connecting vertices in $P_k$ is connected for $k =0, 1.$
\end{enumerate}
Note that there are no edges between vertices in $P_2.$

\subsubsection{}\label{sec:4d-partition} We first remark that $M$ has such a triangulation. If $|L|\to M$ is any triangulation, pass to the first barycentric subdivision $K = L'$ of $L.$ Label the vertices of $L'$ as being barycentres of faces of dimension $i$, for $0 \le i \le 4$. This labelling is independent of the $4$--simplex containing the vertex. Now let $P_0$ be the set of all vertices of $L'$ that are vertices or barycentres of edges in $L$; $P_1$ be the set of all vertices of $L'$ that are barycentres of $2$--faces or $3$--faces in $L$; and $P_2$ be the set of all barycentres of $4$--simplices. This is a partition of the desired form.

\begin{figure}[h]
  \centering
\includegraphics[width=\textwidth]{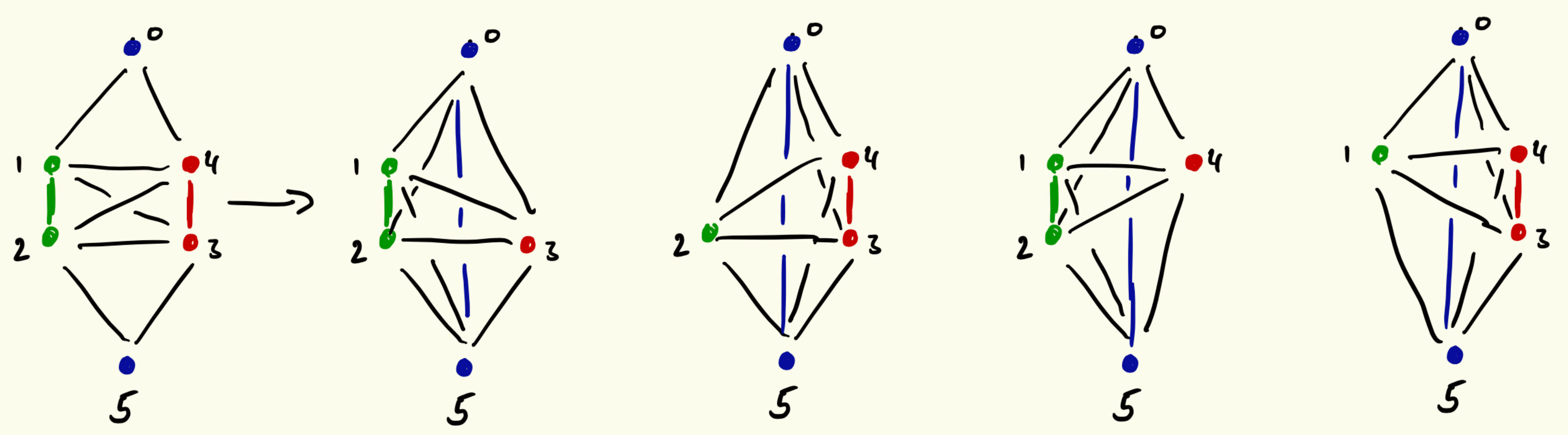}
\caption{Pachner move of type $2-4$}\label{fig:stellar}
\end{figure}

\subsubsection{}\label{sec:4d-stellar} We now apply bistellar operations on $K$ as follows. Each 4--simplex $\sigma$ in $K$ has a unique 3--face $F$ not meeting $P_2,$ and there is a unique 4--simplex $\sigma'$ meeting $\sigma$ in $F.$ The double $4$--simplex $\sigma \cup_F \sigma'$ can be subdivided into four 4--simplices by introducing a new edge $e_F$ between the two vertices in $P_2$ in $\sigma \cup_F \sigma'$ (see Figure~\ref{fig:stellar}). This is a Pachner move of type $2-4$ (cf.\thinspace \cite{Pachner}). Performing this move for every pair of such 4--simplices in $K$ gives a new piecewise linear triangulation $|K_s|\to M,$ with the property that $K$ and $K_s$ have the same vertices, and we do not alter the partition of these vertices. Moreover, all edges of $K$ are edges of $K_s,$ and the only additional edges in $K_s$ are one edge introduced for each double 4--simplex in $K.$ The partition of the vertices of $K$ gives a partition of the vertices of $K_s.$ We will show that this satisfies the following properties:
\begin{enumerate}
\item[$(1_4)$]  Each $4$--simplex has two vertices in two of these sets and one vertex (the \emph{isolated vertex}) in the third.
\item[$(2_4)$]  For each $0 \le i \le  2,$ the graph $\Gamma_i$ consisting of all edges connecting vertices in $P_i$ is connected and contains at least two vertices.
\item[$(3_4)$]  For each $k,$ each 3--face with no vertex in $P_k$ has a 2--face with the property that all but one vertex in the link of the 2--face is in $P_k.$
\item[$(4_4)$]  The degree of each 2--face that meets all three sets in the partition is at least 4.
\end{enumerate}

\begin{figure}[t]
  \centering
 \subfigure[The link condition $(3_4)$]{\includegraphics[height=5.5cm]{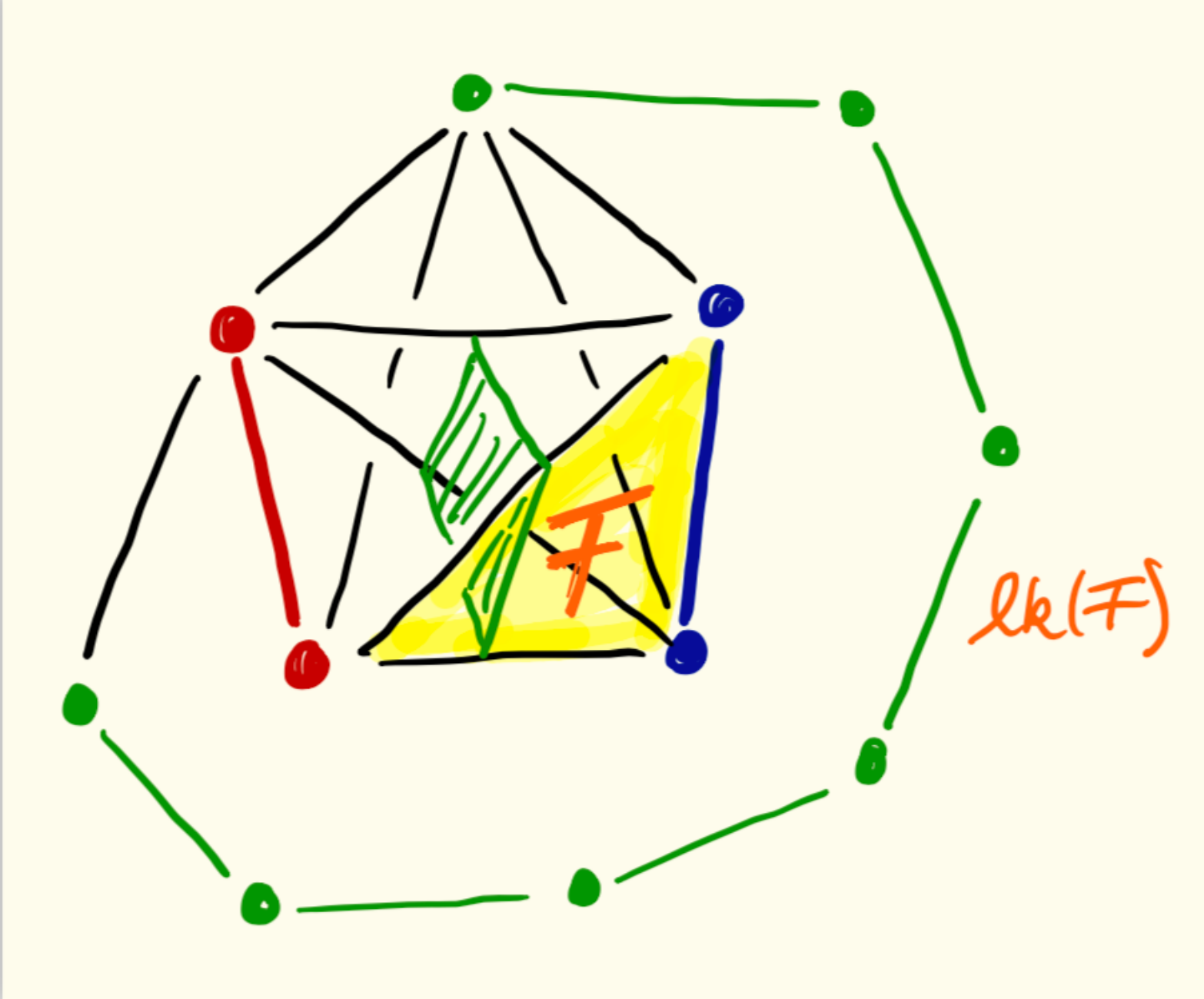}\label{fig:faceF_link}}
    \qquad
     \subfigure[The dual cubical structure of a 2--simplex]{\includegraphics[height=5.5cm]{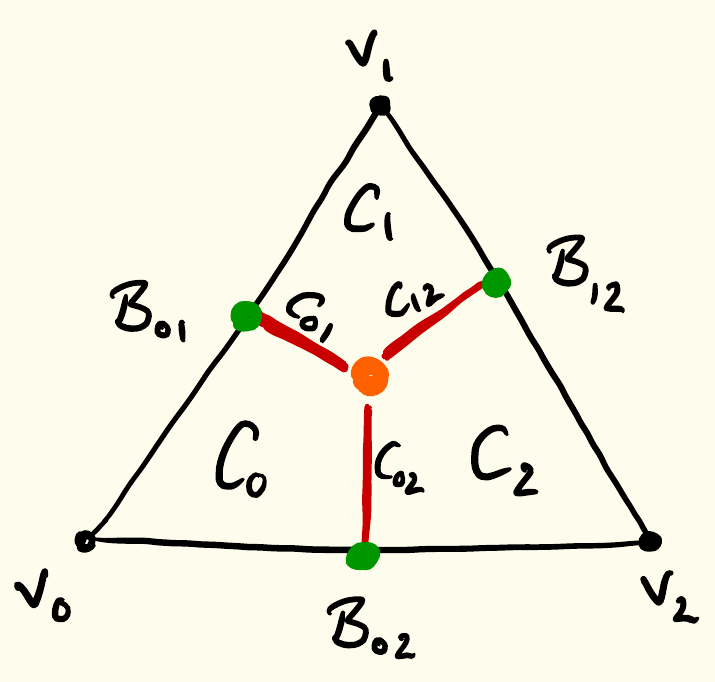}\label{fig:dual cubing 2d}}
\caption{}
\end{figure}

These properties have been singled out, since conditions $(1_4)$, $(2_4)$ and $(3_4)$ will imply all properties of a multisection. These conditions are sufficient but not necessary for obtaining a dual multisection. The last condition gives the additional property that the central submanifold has a non-positively curved cubing. 

\subsubsection{}\label{sec:4d-circlelink}
The triangulation $|K_s|\to M$ satisfies $(1_4)$ and $(2_4)$ by construction. For the next two properties, note that no vertex from the set $P_2$ is isolated since each 4--simplex contains two vertices in $P_2$ due to the bistellar move. Any 3--face $\sigma^3$ with no vertex in $P_k$ has a 1--face $\sigma^1$ with both vertices in $P_2,$ and the remaing vertices in the remaining set $P_j.$ Let $\sigma^2$ be a 2--face of $\sigma^3$ containing $\sigma^1.$ Considering the four 4--simplices incident with $\sigma^1,$ we see that the link of $\sigma^2$ contains a circle $S^1$ triangulated with three 1--simplices having one vertex in $P_j$ and two vertices in $P_k.$ This shows $(3_4).$ To argue that $(4_4)$ holds, first notice that $K_s$ simplicial implies the degree of any 2--face is at least 3. If the 2--face $\sigma^2$ meets all three sets of the partition, then since no vertex from the set $P_2$ is isolated in any 4--simplex of $K_s,$ the degree of $\sigma^2$ is at least 4. 

Hence assume that we have an arbitrary triangulation of $M$ with the property that there is a partition $\{P_0, P_1, P_2\}$ of the set of all vertices satisfying $(1_4)$--$(3_4).$
Define a simplicial map $\phi$ from $M$ to the $2$--simplex $\Delta=\Span\{v_0, v_1, v_2\}$ by sending each vertex in $P_i$ to the vertex $v_i$ of $\Delta$ and extending by affine linear mappings on each $4$--simplex of $\mathcal T.$ Then $\Gamma_i = \phi^{-1}(v_0).$ The dual 1--skeleton of $\Delta$ divides $\Delta$ into three cubes $C_i,$ where $v_i \in C_i.$ Now $\phi^{-1}(C_i) = H_i$ is a regular neighbourhood of $\Gamma_i$ in $M,$ and hence we have a decomposition $M= H_0 \cup H_1 \cup H_2$ into $4$--dimensional 1--handlebodies with pairwise disjoint interiors. Thus, conditions (1) and (2) of Definition~\ref{def:multisection} are satisfied.

\subsubsection{}\label{sec:4d-pl-sub} We next verify that all intersections of the handlebodies are PL submanifolds. First consider $H_{ij}= H_i \cap H_j.$ Since $H_i$ is a regular neighbourhood, it follows that $\partial H_i$ is a 3--dimensional PL submanifold. Now $H_{ij}=H_i \cap H_j = \partial H_i \cap \partial H_j,$ and so the interior of $H_{ij}$ is a 3--dimensional PL submanifold. Since $\partial (H_{ij})$ is collared in $H_{ij},$ it now follows that $H_i \cap H_j$ is a 3--dimensional PL submanifold with boundary. Whence $\partial (H_{ij})= H_0 \cap H_1 \cap H_2 = \Sigma$ is also a PL submanifold.

\subsubsection{}\label{sec:4d-CAT0}  We next show that $(4_4)$ implies that the cubing of $\Sigma$ is non-positively curved. A vertex $v$ of a 2--cube in $\sigma$ is contained on a 2--face $F$ of the triangulation with the property that $F$ meets all three sets $P_0, P_1, P_2$ of the partition. Hence if the degree of $F$ is at least 4, then the degree of $v$ in $\Sigma$ is at least 4, and so the cubing of $\Sigma$ is non-positively curved. 

\subsubsection{}\label{sec:4d-spinedim}  We next show that each $H_{ij}$ has 1--dimensional spine. Here we will use the link condition $(3_4)$. The cubical structure of $\Delta$ consists of the 2--cubes $C_i,$ the 1--cubes $C_{ij} = C_i \cap C_j$ and the 0-cube $C_{ijk} = C_i \cap C_j \cap C_k,$ see Figure~\ref{fig:dual cubing 2d}. We also denote $B_{ij}$ the barycentre of the simplex spanned by $v_i$ and $v_j.$ By the definition of $\Phi$, the natural collapse $C_{ij} \searrow B_{ij}$ lifts to a collapse $H_{ij}  = \Phi^{-1}(C_{ij}) \searrow \Phi^{-1}(B_{ij}).$ Note that $\Phi^{-1}(B_{ij})$ is a union of 1--cubes and 2--cubes. Namely, each 4--simplex with two vertices in $P_i$ and two vertices in $P_j$ meets $\Phi^{-1}(B_{ij})$ in a 2--cube, and each 4--simplex with only one vertex in either $P_i$ or $P_j$ meets $\Phi^{-1}(B_{ij})$ in a 1--cube. The link condition implies that each 2--cube in $\Phi^{-1}(B_{ij})$ has at least one free edge (cf.\thinspace Figure~\ref {fig:faceF_link}), and hence the 2--cube can be collapsed from this edge onto the complementary boundary 1--cubes. Whence $\Phi^{-1}(B_{ij})$ has a 1--dimensional spine, and therefore $H_{ij}$ has a 1--dimensional spine.

\subsubsection{}\label{sec:4d-connected} We complete the proof by showing that each $H_{ij}$ and hence that $\partial (H_{ij}) = \Sigma$ is connected. To be able to use this as an inductive step in the proof of Proposition~\ref{pro:connectivity}, we will only use the fact that the $n$--manifold $M$ has a decomposition $M = H_0 \cup H_1 \cup H_2$ into three compact submanifolds with pairwise disjoint interior, and each $H_i$ and each $H_{ij}$ has a spine of co-dimension at least two. In particular, each connected component of $H_i$ and each connected component of $H_{ij}$ has connected boundary. Let $X$ be a connected component of $H_0.$ If $X \cap H_2 = \emptyset,$ then there is a component $Y$ of $H_1$ such that $\partial X = \partial Y.$ But then $X\cup Y$ is a closed manifold of the same dimension as $M$ and hence $M = X\cup Y.$ This contradicts the fact that $H_2 \neq \emptyset.$ Hence $X \cap H_2 \neq \emptyset,$ and by symmetry $X \cap H_1 \neq \emptyset.$ Now $\partial X$ is a closed $(n-1)$--dimensional manifold and $\partial X = (X \cap H_1) \cup (X\cap H_2).$ Each of $X \cap H_1$ and $X \cap H_2$ is non-empty, $(n-1)$--dimensional and has a spine of co-dimension at least two. Moreover, $(X \cap H_1) \cap (X \cap H_2) = X \cap H_1 \cap H_2,$ which has dimension $n-2,$ and so $X \cap H_1$ and $X \cap H_2$ have disjoint interior. Hence by induction, each of $X \cap H_1,$ $X \cap H_2$ and $X \cap H_1 \cap H_2$ is non-empty, connected and the latter is a closed manifold of dimension $n-2.$ It follows that there is a unique component $Y$ of $H_1$ and a unique component $Z$ of $H_2$ such that $X \cap H_1 = X \cap Y$ and $X\cap H_2 = X \cap Z.$ In particular, $X\cap Y \cap Z$ is non-empty, connected and closed. It follows that $\partial Y = (X \cap Y) \cup (Y \cap Z)$ and $\partial Z = (X \cap Z) \cup (Y \cap Z).$ Whence $X \cup Y \cup Z$ has empty boundary and hence $M = X \cup Y \cup Z.$ This completes the proof.


\subsection{Trisections and multisections of 4--manifolds}
\label{sec:GKvsRT}

We first recall the definition of a $(g,k)$--trisection from \cite{GK}, with the only modification that we state it in the piecewise linear (instead of the smooth) category. Let $Z_k = \natural^k(S^1 \times B^3)$ with $Y_k = \partial Z_k = \sharp^k(S^1 \times S^2).$ Given an integer $g\ge k,$ let $Y_k = Y^-_{k,g} \cup Y^+_{k,g}$ be the standard genus $g$ Heegaard splitting of $Y_k$ obtained by stabilising the standard genus $k$ Heegaard splitting $g-k$ times.

\begin{definition}[(PL variant of Gay-Kirby's definition)]
Given integers $0 \le k \le g,$ a $(g,k)$--trisection of a closed, connected, oriented 4--manifold $M$ is a decomposition of $M$ into three submanifolds $M = H_0 \cup H_1 \cup H_2$ satisfying the following properties:
\begin{enumerate}
\item For each $i=0,1, 2,$ there is a piecewise linear homeomorphism $\Phi\co H_i \to Z_k.$
\item For each $i=0,1, 2,$ taking indices modulo 3, $\Phi(H_i \cap H_{i+1}) = Y^-_{k,g}$ and $\Phi(H_i \cap H_{i-1}) = Y^+_{k,g}.$
\end{enumerate}
\end{definition}

It follows immediately from the definitions that a $(g,k)$--trisection is a multisection of a 4--manifold.
Gay and Kirby~\cite{GK} give two different existence proofs for $(g,k)$--trisections, one using Morse 2--functions and one using handle decompositions. In the first proof, one arranges for the three handlebodies to have the same genus by a homotopy of the Morse map, and in the second this is obtained by adding cancelling pairs of 1-- and 2--handles or 3-- and 4--handles. To connect our multisections to the $(g,k)$--trisections, and hence to prepare for a third existence proof of $(g,k)$--trisections using triangulations, we similarly require a stabilisation result.

As noted in $\cite{GK},$ computing the Euler characteristic of $M$ using the decomposition into handlebodies in a $(g,k)$--trisection gives $\chi(M) = 2+g-3k,$ and hence each of the two constants $g$ and $k$ in the above definition determines the other. 

\begin{lemma}[Stabilisation]\label{lem:stabilisation in dim4}
Let $M$ be a closed, connected, oriented 4--manifold with multisection $M = H_0 \cup H_1 \cup H_2.$ Then the multisection can be modified to a $(g,k)$--trisection with $k = \max(g(H_0), g(H_1), g(H_2)).$
\end{lemma}

\begin{proof}
We first show that if $M = H_0 \cup H_1 \cup H_2$ is a multisection with $g(H_2)=g(H_1)=g(H_0),$ then it is a $(g,k)$--trisection with $g=g(H_0 \cap H_1 \cap H_2)$ and $k = g(H_0).$ Indeed, each pairwise intersection $H_i \cap H_{i+1}$ is a compact connected 3--manifold with 1--dimensional spine, and hence is a 3--dimensional 1--handlebody. It has boundary the surface $\Sigma = H_0 \cap H_1 \cap H_2,$ and hence each of the 3--dimensional handlebodies has the same genus $g$. Any genus $g$ Heegaard splitting of $\partial H_i$ satisfies $g \ge k.$ Building on work of Haken~\cite{Ha}, Waldhausen~\cite{Wald1968} showed that there is a unique PL isotopy class of genus $g$ splittings of $Y_k,$ and hence there is a PL homeomorphism taking the splitting $\partial H_i = (H_i \cap H_{i+1}) \cup_\Sigma (H_i \cap H_{i-1})$ to the stabilised standard splitting $Y_k = Y^-_{k,g} \cup Y^+_{k,g}.$ In particular, we can now extend the map from $\partial H_i \to \partial Z_k$ to a piecewise linear homeomorphism $H_i \to Z_k,$ by first applying radial extensions to the boundaries of a complete set of compression 3--balls of $H_i$ and then radial extensions to the boundary 3--spheres of the complementary 4--balls in $H_i.$

To prove the lemma, it now suffices to show that each multisection of $M$ can be modified to a multisection with all three 4--dimensional 1--handlebodies of the same genus $k = \max(g(H_0), g(H_1), g(H_2)).$ To this end, it suffices to show that we can increase the genus of any one of the handlebodies, say $H_0,$ whilst keeping the genera of the other two handlebodies unchanged. Since $\chi(M) = 2 + g(\Sigma) - g(H_0) - g(H_1) - g(H_2),$ this requires increasing the genus of $\Sigma=H_0 \cap H_1 \cap H_2,$ i.e.\thinspace increasing the genus of one of the handlebodies stabilises the Heegaard splittings of the boundaries of the other two handlebodies. We now describe the required \emph{stabilisation move}.

Let $\alpha$ be an arc properly embedded in the 3--dimensional handlebody $H_1 \cap H_2,$ which is \emph{unknotted}, i.e.\thinspace isotopic into $\partial (H_1 \cap H_2) = \Sigma$ keeping its endpoints on $\Sigma$ fixed. Take a 4--dimensional 1--handle $h_\alpha$ based at $\Sigma \subset \partial H_0$ running along $\alpha$ and add this to $H_0,$ creating a decomposition of $M$ into the three handlebodies $H'_0 = H_0 \cup h_\alpha,$ $H'_1 = \overline{H_1 \setminus h_\alpha}$ and $H'_2 = \overline{H_2 \setminus h_\alpha}$ with pairwise disjoint interior. It follows from our construction that $H_1\searrow H'_1,$ $H_2\searrow H'_2,$ and
$\Sigma' = H'_0 \cap H'_1 \cap H'_2$ is $\Sigma$ with the interior of $h_\alpha \cap \Sigma$ (two open discs) deleted and the link of $h_\alpha \cap (H_1 \cap H_2)$ in $H_1 \cap H_2$ (an annulus) added. We therefore have $g(H'_0) = g(H_0)+1,$ $g(H'_1) = g(H_1),$  $g(H'_2) = g(H_2)$ and $g(\Sigma') = g(H'_0 \cap H'_1 \cap H'_2) = g(\Sigma)+1.$ To show that $\{ H'_0, H'_1, H'_2\}$ is again a multisection, it remains to verify the condition on the pairwise intersections. Since $H'_1 \cap H'_2$ is obtained from the 3--dimensional handlebody $H_1 \cap H_2$ by drilling out an unknotted ark, it is also a 3--dimensional handlebody and hence has a 1--dimensional spine. By our construction, the intersections $H'_0 \cap H'_1$ and $H'_0 \cap H'_2$ can be viewed as the handlebodies $H_0 \cap H_1$ and $H_0 \cap H_2$ with a 1--handle attached to them, so are again 1--handlebodies. This completes the proof.
\end{proof}

\begin{figure}[t]
  \centering
\includegraphics[height=5cm]{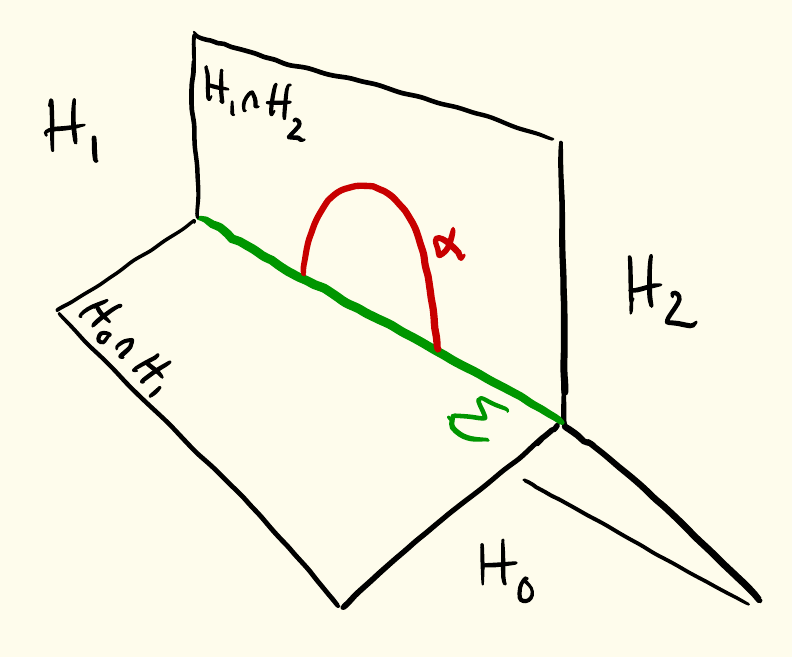}
 \caption{Stabilisation of a multisection}
\end{figure}

The proof of the above lemma shows that there is a natural \emph{stabilisation move}, which allows us to increase the genus of any of the 4--dimensional 1--handlebodies in a multisection by one. (This stabilisation generalises to non-orientable manifolds and higher dimensions.) The lemma shows that the existence proof in dimension 4 of multisections (in the sense of Definition~\ref{def:multisection}) in \S\ref{sec:construction dim 4} implies the existence of $(g,k)$--trisections (in the sense of Gay and Kirby). In light of this 
\begin{center}
\emph{from now onwards, we will simply call a multisection in dimension 4 a trisection,}
\end{center}
even if the 4--dimensional handlebodies do not have the same genus, or the manifold is non-orientable.


\subsection{Generalised multisections and recursive structure}
\label{sec:recursive}

Before we give the existence proof in all dimensions, we give a preliminary result, which will be used in the existence proof and is of independent interest. In particular, it foreshadows our discussion of generalisations of a multisection. 

Inside a multisection of an $n$-dimensional manifold, we see a stratification of the boundary of each handlebody into lower dimensional manifolds. For example, for a trisection where $n=4,5$, we see partition functions on the boundaries of each handlebody, dividing the boundary into two pieces. For a quadrisection, where $n=6,7$, the boundaries of the handlebodies are divided into three pieces. However, the top dimensional pieces are not necessarily handlebodies, whereas all the pieces have spines of low dimension. The same works in all dimensions. Namely for $n$--manifolds, with $n=2k$ or $n=2k+1$, the boundaries of the handlebodies have natural divisions into $k$ regions. Each of these regions has a spine of dimension at most two. However, these regions are not necessarily 1--handlebodies. So these are not multisections in the sense defined in this paper.

The following definition applies to the subdivisions of the multisection submanifolds in the recursive structure. Note that there is no relationship assumed between $n$ and $k.$

\begin{definition}[(Generalised multisection of closed manifold)]\label{def:generalised multisection}
Let $M$ be a closed, connected, piecewise linear $n$--manifold. A \emph{generalised multisection} of $M$ is a collection of $k+1$ piecewise linear $n$--dimensional submanifolds $H_i \subset M,$ where $0 \le i \le k,$ subject to the following three conditions:
\begin{enumerate}
\item Each $H_i$ is non-empty and has a spine of codimension at least two.
\item The submanifolds $H_i$ have pairwise disjoint interior, and $M = \bigcup_i H_i.$
\item  The intersection $H_{i_1} \cap H_{i_2} \cap \ldots \cap H_{i_r}$ of any proper subcollection of the submanifolds ($r \le k$) is a compact submanifold with boundary and of dimension $n-r+1.$ Moreover, it has a spine of codimension at least two.
\end{enumerate}
\end{definition}

What is used in this paper for the general existence proof of multisections is the following relationship between the properties of having low dimensional spines and connectivity of the intersection submanifolds. 

\begin{proposition}\label{pro:connectivity}
Suppose that a closed connected manifold $M$ has a generalised multisection into submanifolds $H_i$ for $0 \le i \le k.$ Then the intersection $H_{i_1} \cap H_{i_2} \cap \ldots \cap H_{i_r}$ of \emph{any} collection of the submanifolds is non-empty and connected. In particular, each $H_i$ and the intersection $H_{0} \cap H_{1} \cap \ldots \cap H_{k}$ is connected. Moreover, $H_{0} \cap H_{1} \cap \ldots \cap H_{k}$ is a closed manifold of dimension $n-k.$
\end{proposition}

\begin{proof}
The argument is by complete induction on $k$. To start the induction, suppose $M$ is a manifold of dimension $n$ and $k=1$. Whence $M = H_0 \cup H_1.$ Each component $X$ of $H_0$ has a spine of codimension 2 and hence connected boundary. Since $M = H_0 \cup H_1$ and $\Int(H_0)\cap \Int(H_1) = \emptyset,$ we have $\partial H_0 = \partial H_1.$ Therefore there is a component $Y$ of $H_1$ such that $\partial X = \partial Y.$ But then $X\cup Y$ is a closed $n$--manifold and hence $M = X\cup Y.$ Moreover, $X \cap Y = \partial X$ is a non-empty, closed and connected manifold of dimension $n-1.$ This proves the results for all manifolds $M$ and all decompositions $M = H_0 \cup H_1.$

The induction step for $k=2$ follows \S\ref{sec:4d-connected} verbatim.
We therefore give the general induction step without further ado.
Hence assume the conclusion holds for all manifolds and all multisections into at most $k_0$ submanifolds. Assume that we are given a generalised multisection of an $n$--manifold $M$ with $k_0+1$ submanifolds $H_0, \ldots, H_{k_0}.$ Each component $X_0$ of $H_0$ has $\partial X_0$ connected. As above, we have 
$$\partial X_0 = (X_0 \cap H_1) \cup \ldots \cup (X_0 \cap H_{k_0}).$$
This is a generalised multisection of the closed manifold $\partial X_0$ with at most $k_0$ submanifolds. By the induction hypothesis, all components in the above decomposition are connected, hence there are unique components $X_i$ of $H_i$ such that $X_0 \cap H_i = X_0 \cap X_i,$ where we put $X_i = \emptyset$ if $X_0 \cap H_i = \emptyset.$ Whence 
$$\partial X_0 = (X_0 \cap X_1) \cup \ldots \cup (X_0 \cap X_{k_0}).$$
Since the non-empty submanifold $X_0 \cap X_i \cap X_j$ is contained in $X_i \cap X_j,$ it follows by uniqueness that 
$$\partial X_i = (X_i \cap X_0) \cup (X_i \cap X_1) \cup \ldots \cup (X_i \cap X_{k_0}),$$
where we omit $X_i \cap X_i$ from the union. It follows that $X_0 \cup X_1 \cup \ldots \cup X_{k_0}$ has empty boundary, and hence equals $M.$ In particular, each $X_i \neq \emptyset.$ This completes the proof.
\end{proof}


\subsection{Multisections of closed $(2k+1)$--manifolds}
\label{sec:construction dim odd}

\begin{theorem}\label{thm:existence odd}
Every closed, connected, odd dimensional PL manifold has a multisection.
\end{theorem}

\begin{proof}
The proof is a generalisation of the construction for 3--dimensional manifolds, using some of the arguments given in the 4--dimensional case.
Assume $M$ has dimension $n=2k+1\ge 3$ and the piecewise linear triangulation $|K|\to M.$ Denote $K'$ the first barycentric subdivision of $K$ and partition the vertices of $K'$ into sets $P_0, P_1, \ldots, P_{k}$ as follows. The set $P_i$ contains all vertices of $K'$ that are the barycentres of $2i$--simplices or $(2i+1)$--simplices in $K.$ 

Now define a simplicial map $\phi\co |K'| \to \sigma,$ where $\sigma$ is a $k$--simplex, by mapping $P_i$ to the $i$\textsuperscript{th} vertex of $\sigma.$ This defines a piecewise linear map $\phi \co M \to \sigma.$ Each $n$--simplex in $K'$ meets each set $P_j$ in precisely two vertices, and since $M$ is connected, the graph $\Gamma_i$ in the 1--skeleton of $K'$ spanned by all vertices in $P_j$ is connected. We identify $M$ with $|K'|.$ Any regular neighbourhood of $\Gamma_i$ is an $n$--dimensional 1--handlebody in $M.$

Consider the cubical cell decomposition of $\sigma$ arising from the dual spine $\Pi^{k-1} \subset \sigma.$ This has $k+1$ $k$--cubes, which meet in pairs along $(k-1)$--cubes. The pull-back of this decomposition divides each $n$--simplex in $M$ into regions by the inverse images of the $k$--cubes, and, moreover, the pre-image of the cube $C_i$ containing the $i$\textsuperscript{th} vertex of $\sigma$ is a regular neighbourhood of $\Gamma_i.$ In particular, letting $H_i = \phi^{-1}(C_i)$ gives a decomposition of $M$ into $k+1$  1--handlebodies satisfying (1) and (2) in Definition~\ref{def:multisection}.

We claim that (3) and (4) are also satisfied. First note that the argument given in \ref{sec:4d-pl-sub} can be iterated to show that all intersections are PL submanifolds of the stated dimensions, and that all but the central submanifold have non-empty boundary. It follows from Proposition~\ref{pro:connectivity} that it remains to prove the claim in (3) about the codimensions of the spines. The key is the partition map $\phi\co M \to \sigma.$

We first study the restriction $\phi\co\Delta \to \sigma$, where $\Delta$ is a $(2k+1)$--simplex of $M$. The first claim is that $\phi^{-1}(x)$ is a cube of dimension  $k+1-j$, where $x \in \sigma$ is in the interior of a face of codimension $j$ in $\sigma$. We choose affine coordinates in the simplices $\Delta$ and $\sigma$ so that the coordinates are ordered in the following way. In $\Delta$, the first two coordinates $a_0,b_0$ are for vertices in $P_0$, the second two $a_1,b_1$ for vertices in $P_1$, and so forth. Similarly in $\sigma$ the coordinates are ordered to correspond to the vertices labelled $0,1, \dots k$. Then  $\phi(a_0,b_0,a_1,b_1,\ldots,  a_k, b_k)=(a_0+ b_0,a_1 +b_1, \ldots, a_k + b_k)$ using these coordinates. The $a_i,b_i$ satisfy $0 \le a_i,b_i \le 1$ and $\Sigma_{i=0} (a_i +b_i) =1$. If the image point $x \in \Delta$ lies in the interior of a codimension $j$ face, then $j$ of its coordinates are $0$, and the remaining coordinates have fixed non-zero values. Hence the inverse image is a $(k+1-j)$--cube as claimed. 

In particular, the central submanifold $H_0 \cap \ldots \cap H_k$ has a natural cubing since it meets each $n$--simplex $\Delta \subset M$ in a single $(k+1)$--cube, which is the pre-image of the barycentre $B \in \sigma.$ It will be shown in the proof of Theorem~\ref{thm:CAT(0)} in \S\ref{sec:CAT(0)} that this cubing is in fact non-positively curved.

Now consider the intersection $G = H_{i_1} \cap H_{i_2} \cap \ldots \cap H_{i_r},$ where $2\le r\le k.$ By construction, $C = \phi(G)$ is a $(k-r+1)$--cube in $\Pi^{k-1}\subset \sigma.$ The cube $C$ naturally collapses onto the barycentre $B'$ of the subsimplex of $\sigma$ with vertex set corresponding to $\{i_1, \ldots, i_r\}.$ This has codimension $k-r+1.$
By construction of the partition map $\phi\co M \to \sigma,$ this collapse lifts to a collapse of $G$ onto the pre-image of $B'.$ This meets every top-dimensional simplex in $M$ in a cube of dimension $k+1-(k-r+1)=r.$ Whence $G$ has a spine with a cubing by $r$--cubes, giving the claimed dimension.
\end{proof}

\begin{scholion}
Every closed, connected, PL manifold of dimension $2k+1$ has a multisection with the property that each intersection $H_{i_1} \cap H_{i_2} \cap \ldots \cap H_{i_r}$ for $1\le r \le k$ has a spine with a cubing by $r$--cubes and the intersection of all $k+1$ handlebodies has a non-positively curved cubing by $(k+1)$--cubes.
\end{scholion}


\subsection{Multisections of closed $2k$--manifolds}
\label{sec:construction dim even}

\begin{theorem}\label{thm:existence even}
Every closed, connected, even dimensional PL manifold has a multisection.
\end{theorem}

\begin{proof}
Assume $M$ has dimension $n=2k\ge 4$ and the piecewise linear triangulation $|L|\to M.$ Denote $L'$ the first barycentric subdivision of $L$ and partition the vertices of $L'$ into sets $P_0, P_1, \ldots, P_{k}$ as follows. For $0\le i \le k-1$ the set $P_i$ contains all vertices of $L'$ that are the barycentres of $2i$--simplices or $(2i+1)$--simplices in $L.$ The set $P_k$ contains the barycentres of the $2k$--simplices.
Whence each $n$--simplex has two vertices in the partition sets $P_0, P_1, \ldots, P_{k-1}$ and a single vertex in $P_k$. 

As in \ref{sec:4d-stellar}, each $n$--simplex $\sigma$ of $L'$ has a unique $(n-1)$--face $F$ not meeting $P_k$ and there is a unique  $n$--simplex $\sigma'$ meeting $\sigma$ in $F.$ We again subdivide the double simplex $\sigma\cup_F\sigma'$ by performing a Pachner move of type $2-n,$ which introduces a new edge between the two vertices of the double simplex that are in $P_k$  and replaces the two $n$--simplices by $n$ $n$--simplices. Applying this to all such double simplices in $L'$ gives a new PL triangulation $|K|\to M$ of $M.$ As before, there is a natural identification of the vertex sets of $K$ and $L'$, and we maintain the partition. Whence the graphs $\Gamma_i$ spanned by $P_i$ in $K$ are all connected.


Now define a simplicial map $\phi\co |K| \to \sigma,$ where $\sigma$ is a $k$--simplex, by mapping $P_i$ to the $i$\textsuperscript{th} vertex of $\sigma.$ As before, we study the restriction of this to an $n$--simplex $\Delta$ in $M.$ There is a unique partition set $P_h$ that $\Delta$ meets in only one vertex.
We can choose affine coordinates on $\Delta$ such that 
\begin{multline*}
(a_0,b_0, \ldots, a_{h-1}, b_{h-1}, a_{h}, a_{h+1}, b_{h+1}, \ldots, a_k, b_k)\\ 
\mapsto \; (a_0+b_0, \ldots,a_{h-1}+b_{h-1}, a_h, a_{h+1}+b_{h+1}, \ldots, a_{k} +b_{k}).
\end{multline*}
The inverse image of a point $x\in \sigma$ is a cube whose dimension depends on the co-dimension of the face $F$ containing $x$ in its interior and on whether or not $F$ contains the $h$\textsuperscript{th} vertex of $\sigma.$ Suppose $F$ has co-dimension $j.$
If $F$ does not contain the $h$\textsuperscript{th} vertex of $\sigma,$ then $\phi^{-1}(x)$ is a $(k+1-j)$--cube since $2j-1$ coordinates in the preimage must equal zero. In contrast, if $F$ contains the $h$\textsuperscript{th} vertex of $\sigma,$ then $\phi^{-1}(x)$ is a $(k-j)$--cube since $2j$ coordinates in the preimage are zero and the coordinate corresponding to the singleton is a non-zero constant.

In particular, the central submanifold $H_0 \cap \ldots \cap H_k$ has a natural cubing since it meets each $n$--simplex $\Delta \subset M$ in a single $k$--cube, which is the pre-image of the barycentre $B \in \sigma.$ 

According to Proposition~\ref{pro:connectivity}, it remains to verify that the submanifolds have spines of the desired dimensions.

As in the odd dimensional case, we now consider the intersection $G = H_{i_1} \cap H_{i_2} \cap \ldots \cap H_{i_r},$ where $2\le r\le k.$ By construction, $C = \phi(G)$ is a $(k-r+1)$--cube in $\Pi^{k-1}\subset \sigma.$ The cube $C$ naturally collapses onto the barycentre $B'$ of the subsimplex of $\sigma$ with vertex set corresponding to $\{i_1, \ldots, i_r\}.$ This has codimension $k-r+1.$
By construction of the partition map $\phi\co M \to \sigma,$ this collapse lifts to a collapse of $G$ onto the pre-image of $B'.$ This meets a top-dimensional simplex $\Delta$ in $M$ either in a cube of dimension $k-(k-r+1)=r-1$ or in a cube of dimension $k+1-(k-r+1)=r,$ depending on whether the singleton of $\Delta$ is in a partition set corresponding to $\{i_1, \ldots, i_r\}$ or not. Whence $G$ has a spine with a cubing by $r$--cubes and $(r-1)$--cubes, giving the claimed dimension $r.$ 

We now make the stronger, additional observation that each of the $(r-1)$--cubes in the spine for $G$ is a boundary face of some $r$--cube unless $r\ge k.$ Whence suppose $\Delta$ is an $n$--simplex in $M$ with the property that $\phi^{-1}(B')$ meets $\Delta$ in an $(r-1)$--cube. Let $P_h$ be the partition set containing the singleton of $\Delta,$ so $h\in \{i_1, \ldots, i_r\}.$ Due to the barycentric subdivision and the Pachner move, $\Delta$ meets $P_k$ in precisely two vertices; whence $k\neq h.$ Now $\Delta$ is an $n$--simplex obtained from a Pachner move on a double $n$--simplex, which contains exactly two vertices from each partition set. So $\Delta$ contains all vertices of this double $n$--simplex except for one vertex, say $v$, that is in the partition set $P_h.$ If there is a vertex $w$ of $\Delta$ with $w\in P_m$ and $m\notin \{i_1, \ldots, i_r, k\},$ then there is an $n$--simplex in $K$ with vertex set $(\Delta^{(0)}\setminus\{w\}) \cup P_h.$ This meets $\phi^{-1}(B')$ in an $r$--cube which has the $(r-1)$--cube $\phi^{-1}(B')\cap\Delta$ as a face. Hence suppose there is no such vertex $w.$ This implies that either $\{i_1, \ldots, i_k\} = \{ 0, \ldots, k\}$ or $\{i_1, \ldots, i_k\} = \{ 0, \ldots, k-1\}.$ But then either $r = k+1$ or $r=k$ and $k\notin \{i_1, \ldots, i_k\}.$ 

To conclude the proof of the third property it suffices to show that if $r=k,$ then each $r$--cube in the spine of $G$ has one of its boundary $(r-1)$--cubes as a free facet, and hence can be collapsed onto the union of its remaining top-dimensional facets. In this case, the set $\{i_1, \ldots, i_k\}$ does not contain a single vertex of the target simplex $\sigma.$ Suppose $h \notin \{i_1, \ldots, i_k\}.$ If the preimage of $B'$ meets $\Delta$ in an $r$--cube, then the singleton of $\Delta$ is in the set $P_h.$ Due to the barycentric subdivision and the Pachner move, $\Delta$ meets $P_k$ in precisely two vertices; whence $k\neq h.$ Now $\Delta$ is an $n$--simplex obtained from a Pachner move on a double $n$--simplex, which contains exactly two vertices from each partition set. So $\Delta$ contains all vertices of this double $n$--simplex except for one vertex that is in the partition set $P_h.$ Let $m\in \{i_1, \ldots, i_k\}\setminus \{k\}$ and consider the co-dimension two facet $\Delta^{n-2}$ of $\Delta$ which does not meet $P_m.$ From the double $n$--simplex and the fact that the triangulation is PL, it follows as in \ref{sec:4d-circlelink} that the link of $\Delta^{n-2}$ is a circle triangulated with three 1--simplices having one vertex in $P_h$ and two vertices in $P_m.$ So the $(n-2)$--simplex $\Delta^{n-2}$ is contained in precisely three $n$--simplices: $\Delta$ is obtained by adding the two vertices in $P_m,$ and the other two are obtained by adding one vertex in $P_m$ and one vertex in $P_h.$ Whence the boundary $(r-1)$--cube $\phi^{-1}(B')\cap \Delta^{n-2}$ of the $r$--cube $\phi^{-1}(B')\cap \Delta$ is not a boundary face of another $r$--cube. This completes the proof of the theorem.
\end{proof}

\begin{scholion}
Every closed, connected, PL manifold of dimension $2k$ has a multisection with the property that each intersection $H_{i_1} \cap H_{i_2} \cap \ldots \cap H_{i_r}$ for $1\le r \le k-1$ has a spine with a cubing by $r$--cubes, each intersection of $k$ handlebodies has a spine with a cubing by $(k-1)$--cubes and the intersection of all $k+1$ handlebodies has a cubing by $k$--cubes.
\end{scholion}


\section{Structure results}

It is first shown that triangulations and partition maps can be chosen such that the central submanifold has a non-positively curved cubing. Next, the structure result for finitely presented groups is discussed. The section concludes with a description of the higher homotopy groups.

\subsection{Non-positively curved cubings from multisections}
\label{sec:CAT(0)}

We work with the combinatorial definition of a non-positively curved cubing (see \cite[\S2.1]{Wi}). A \emph{flag complex} is a simplicial complex with the property that each subgraph in the 1--skeleton that is isomorphic to the 1--skeleton of a $k$--dimensional simplex is in fact the 1--skeleton of a $k$--dimensional simplex. A cube complex is \emph{non-positively curved} if the link of each vertex is a flag complex. Here, the link of a vertex in a cube complex is the simplicial complex whose $h$--simplices are the corners of $(h+1)$--cubes adjacent with the vertex. The main facts we will need are that the barycentric subdivision of any complex is flag, and that the link (in the sense of simplicial complexes) of any simplex in a flag complex is a flag complex.

\begin{reptheorem}{thm:CAT(0)}
Every piecewise linear manifold has a triangulation supporting a multisection such that the central submanifold has a non-positively curved cubing.
\end{reptheorem}

\begin{proof}
Suppose the dimension is $2k+1$, so $M$ is mapped to the $k$--simplex $\sigma.$ As in the existence proof, our starting point is a triangulation $|K|\to M$ that is a first barycentric subdivision and the canonical partition map $\Phi\co M \to \sigma.$ Since $K$ is a first barycentric subdivision, it is a flag complex.

The central submanifold $\Sigma = H_0 \cap \ldots \cap H_k$ is the pre-image of the barycentre $B$ of $\sigma$ and meets each top-dimensional simplex $\Delta$ in a $(k+1)$--cube. A corner $c$ of such a $(k+1)$--cube $C$ in $\Sigma$ is the barycentre of a $k$--simplex $\Delta^k$ in $M$ meeting all sets of the partition. So the $2^{k+1}$ vertices of $\phi^{-1}(B)\cap \Delta$ correspond to the $2^{k+1}$ barycentres of the $k$--faces of $\Delta$ meeting each set of the partition. We claim that there is a simplicial isomorphism $\lk_\Sigma(c) = \lk_{K}(\Delta^k).$ The latter is flag since it is the link of a simplex in a flag complex, and so this claim implies the conclusion of the theorem.

Suppose $c \in \Delta^k \subset \Delta^{2k+1}.$ Then the $k+1$ vertices of $\Delta^{2k+1}$ not in $\Delta^k$ meet each set in the partition, and $\lk_{K}(\Delta^k)$ is spanned by these vertices as $\Delta^{2k+1}$ ranges over all top-dimensional simplices containing $\Delta^k.$ In particular, each $k$--simplex in $\lk_{K}(\Delta^k)$ meets each set in the partition, and each simplex in $\lk_{K}(\Delta^k)$ has no two vertices in the same partition set. Now each edge of such a $\Delta^{2k+1}$ between vertices in the same partition set corresponds to an edge in the central submanifold with an endpoint at $c$ and vice versa. Hence there is a bijection between vertices in $\lk_\Sigma(c)$ and $\lk_{K}(\Delta^k).$ The same holds for all higher dimensional cells in the links: an $h$--simplex in $\lk_{K}(\Delta^k)$ corresponds to an $h$--simplex $\Delta_0^h$ in some $\Delta_0^{2k+1}\supset \Delta^k.$ Then $\Delta_0^h$ and $\Delta^k$ span a subsimplex of $\Delta_0^{2k+1}$ which meets the central submanifold in an $(h+1)$--cube. Whence we obtain an injective simplicial map $\lk_{K}(\Delta^k)\to \lk_\Sigma(c),$ which is clearly bijective. This completes the proof for manifolds of odd dimension.

Now suppose the dimension of $M$ is $n=2k.$ For the triangulation constructed in the existence proof, it turns out that the cube complexes for $k>2$ do not satisfy the flag consition due to the $2-n$ Pachner move. We therefore use a different construction, which results in a greater number of simplices. Start with any piecewise linear triangulation $|L|\to M.$ Denote $L'$ the first barycentric subdivision of $L.$ All dual 1--cycles in $L'$ have even length, and hence the dual 1--skeleton is a bipartite graph. We therefore have a partition of the $n$--simplices into two sets $V_0$ and $V_1$ such that $n$--simplices in $V_0$ only meet $n$--simplices in $V_1$ along their codimension-one faces and  vice versa. Now let $K = L''$ be the second barycentric subdivision of $L.$ We define a partition $\{P_0, \ldots, P_k\}$ of the vertices of $K$ as follows. The set $P_0$ consists of all vertices of $L'$ and all barycentres of $n$--simplices in $V_0.$ The set $P_1$ consists of all barycentres of edges in $L'$ and all barycentres of $n$--simplices in $V_1.$ For each $2\le j \le k,$ the set $P_j$ consists of all barycentres of $(2j-2)$--simplices and $(2j-1)$--simplices.

We claim that the graph $\Gamma_j$ spanned by all vertices in $P_j$ is connected for each $j.$ First consider the set $P_j$ for any $j>1.$  Then each vertex in $P_j$ is the barycentre of a $(2j-2)$--simplex or a $(2j-1)$--simplex in $L'.$ Any $(2j-2)$--simplex in $L'$ is in the boundary of a $(2j-1)$--simplex, and the dual graph of the $(2j-1)$--skeleton is connected. This shows that the graph $\Gamma_j$ is connected.

Now consider $P_0$ for $j=0$ or $1.$ Each element in $P_0$ is either a vertex in $L'$ or the barycentre of some $n$--simplex in $L'.$ Any vertex of $L'$ is connected in $\Gamma_0$ to the barycentre of some $n$--simplex in $L'.$
Any two barycentres of $n$--simplices in $L'$ that are in $P_0$ are connected by a path in the dual 1--skeleton, and the vertices in this path alternate between vertices in $P_0$ and $P_1.$ Now any two adjacent $n$--simplices share at least $n$ vertices, and any two $n$--simplices at distance two in the dual 1--skeleton share at least $n-2$ vertices. These vertices are all in $P_0$ and connect to the barycentre of any $n$--simplex in $L'$ that is in $P_0.$ Whence $\Gamma_0$ is connected.
The argument that $\Gamma_1$ is connected is similar, by observing that any two $n$--simplices at distance two in the dual 1--skeleton share at least $n+1 \choose 2$ edges. This completes the argument that each graph $\Gamma_j$ is connected. 

As in the existence proof, to show that the associated partition function defines a multisection of $M,$ we need to establish that the dimension of the spine drops by one when $r=k.$ Hence suppose $G = H_{i_1} \cap H_{i_2} \cap \ldots \cap H_{i_k}.$ From the computation of the dimensions of the pre-images in the proof of Theorem~\ref{thm:existence even}, we know that if the singleton of the $n$--simplex $\Delta$ is mapped to a vertex in $\{i_1, \ldots, i_k\},$ then we obtain a $(k-1)$--cube in the spine. A singleton is either mapped to $h=0$ or $1.$ So suppose the singleton of $\Delta$ is mapped to $h$ and $h \notin  \{i_1, \ldots, i_k\},$ so that the spine for $G$ meets $\Delta$ in a $k$--cube. Then the $(n-1)$--facet $\Delta^{n-1}$ with no vertex mapping to $h$ contains this $k$--cube. Due to the barycentric subdivision construction, the link of $\Delta^{n-1}$ is a circle triangulated with four vertices, and precisely three of these map to $h$ and one maps to $\{0, 1\} \setminus \{h\}.$ But then the $k$--cube has a free $(k-1)$--face and hence can be collapsed. (In dimension four, this is precisely condition $(3_4)$ in \S\ref{sec:4d-stellar}.) Whence we have shown that the intersection of any $k$ handlebodies has a spine of dimension $k-1.$

Having established that the given triangulation and partition function define a multisection, the proof given above for the odd-dimensional case now applies verbatim to show that the link of an arbitrary corner in the cubing is simplicially isomorphic to the link of the simplex in the triangulation $L''$ having the corner as its barycentre.
\end{proof}


\subsection{Structure of the fundamental group}
\label{sec:fundamental group}

We explicitly analyse the structure of the fundamental group in the case of a trisection of a 4-manifold, as this already gives a result about all finitely presented groups. We then note how higher dimensional multisections give similar results. In particular, Theorem~\ref{thm:CAT(0)} implies that one may choose as the \emph{central} group in the associated generalised graph of groups the fundamental group of a manifold with a non-positively curved cubing.

\begin{repproposition}{pro:structure of fund gp}
Every finitely presented group has a generalised graph of groups decomposition as shown in Figure~\ref{fig:structure of groups_intro}, where the vertex groups $\Gamma_0, \Gamma_1, \Gamma_2$ are free of rank $\le g,$ $\Gamma_g$ is the fundamental group of a closed orientable surface of genus $g,$ all edge groups are naturally isomorphic with $\Gamma_g$ and the oriented edges represent epimorphisms.
\end{repproposition}

\begin{proof}
Given a finitely presented group $\Gamma,$ there is a closed, orientable PL 4--manifold $M$ with $\pi_1(M)\cong \Gamma.$ Choose a trisection $M = H_0 \cup H_1 \cup H_2$ of $M.$ Then the central submanifold $\Sigma$ is a closed, orientable surface. Write $g = g(\Sigma).$ The inclusions $i_k \co \Sigma \subset H_k$ induce epimorphisms $\iota_k\co \Gamma_g = \pi_1(\Sigma) \to \pi_1(H_k) = \Gamma_k.$ We now define a generalised graph of groups by choosing as vertex groups the four groups $\Gamma_g,$ $\Gamma_0,$ $\Gamma_1$  and $\Gamma_2,$ and as edge groups three copies of $\Gamma_g$ with edge maps the identity map $\psi^-_k \co \Gamma_g \to \Gamma_g$ and the epimorphisms $\psi^+_k \co  \Gamma_g \to \Gamma_k.$ Denote $\pi_1(\mathcal{G})$ the fundamental group of the generalised graph of groups. We claim that $\pi_1(\mathcal{G})\cong \pi_1(M).$

To this end, we compute $\pi_1(M)$ through two applications of the generalised Van Kampen Theorem~\cite[Theorem 6.2.11]{Geo}. First consider $\pi_1(H_0 \cup H_1).$ This is the push-out
$$\pi_1(H_0) \twoheadleftarrow \pi_1(H_0 \cap H_1) \twoheadrightarrow \pi_1(H_1).$$
Using the epimorphism $\pi_1(\Sigma) \twoheadrightarrow \pi_1(H_0 \cap H_1),$ we also obtain $\pi_1(H_0 \cup H_1)$ as the push out of the induced maps
$$\pi_1(H_0) \twoheadleftarrow \pi_1(\Sigma) \twoheadrightarrow \pi_1(H_1).$$
Now $M = (H_0 \cup H_1) \cup H_2$ gives the push-out
$$\pi_1(H_0\cup H_1) \twoheadleftarrow \pi_1((H_0 \cup H_1)\cap H_2) \twoheadrightarrow \pi_1(H_2),$$
and using the induced map $\pi_1(\Sigma) \twoheadrightarrow \pi_1((H_0 \cup H_1)\cap H_2)$ gives desired isomorphism.
\end{proof}

\begin{theorem}\label{thm:group}
Let $M$ be a closed, connected, piecewise linear $n$--manifold with a multisection $\{H_i\},$ where $0 \le i \le k$ and $n=2k$ or $n=2k+1.$ 

Then $\pi_{1}( H_0 \cap H_1 \cap \ldots \cap H_k) \to \pi_{1}(M)$ is onto and $\pi_{1}(\;\bigcup (H_{i_1} \cap H_{i_2} \cap \ldots \cap H_{i_{k}})\;) \to \pi_{1}(M)$ is an isomorphism, where the union is taken over all intersections of $k$ pairwise distinct handlebodies.
\end{theorem}

\begin{proof}
If one takes the intersection of at most $k$ handlebodies, one obtains a manifold with spine of co-dimension at least 2. Moreover, all 
such intersections with at most $k-1$ handlebodies have a spine of co-dimension at least 3. So by general position we can push loops off the first type of components and the discs defining relations of the fundamental group can be pushed off the second type of components. This proves the claim.
\end{proof}

\begin{corollary}
 Let $M$ be a closed, connected, piecewise linear $n$--manifold with a multisection. Then the multisection lifts to a multisection of any covering space of $M$.
\end{corollary}

\begin{proof}
By Theorem~\ref{thm:group}, this follows immediately, since all the intersection components of the muitisection lift to connected components, as they have fundamental group which maps onto the fundamental group of $M$.
\end{proof}

We remark that the corollary also follows from our existence proof by lifting a triangulation supporting a trisection to the covering space.


\subsection{Recursive structure and higher homotopy groups}
\label{sec:Recursive structure and higher homotopy groups}

As an extension of the results in Section~\ref{sec:recursive} and Section~\ref{sec:fundamental group}, we investigate how the higher homotopy groups of a multisection are carried by the multisection submanifolds and the central submanifold. 

\begin{theorem}  Let $M$ be a closed, connected, piecewise linear $n$--manifold with a multisection $\{H_i\},$ where $0 \le i \le k$ and $n=2k$ or $n=2k+1.$ Then the homotopy groups of $M$ satisfiy the following conditions;
\begin{enumerate}
\item  If $n=2k$, then $\pi_{2j+2}(\bigcup (H_{i_1} \cap H_{i_2} \cap \ldots \cap H_{i_{k-j}})) \to \pi_{2j+2}(M)$ is onto for $0 \le j \le k-2.$ Moreover, $\pi_{m}(\bigcup (H_{i_1} \cap H_{i_2} \cap \ldots \cap H_{i_{k-j}})) \to \pi_{m}(M)$ is an isomorphism for $m \le 2j+1$ and $0 \le  j \le k-2.$
\item  If $n=2k+1$, then $\pi_{2j+3}(\bigcup (H_{i_1} \cap H_{i_2} \cap \ldots \cap H_{i_{k-j}})) \to \pi_{2j+3}(M)$ is onto for $0 \le j \le k-2$. Moreover, $\pi_{m}(\bigcup (H_{i_1} \cap H_{i_2} \cap \ldots \cap H_{i_{k-j}})) \to \pi_{m}(M)$ is an isomorphism for $m \le 2j+2$ and $0 \le  j \le k-2.$
\end{enumerate}
\end{theorem}

\begin{proof}
The proof is an exercise in transversality, using the codimensions of the spines of the multisection submanifolds.
In particular, an intersection of $k-j$ of the handlebodies has a spine of codimension at least $n-2k+2j+1,$ except in the special case of $n=2k$ and $j=0$, in which case the spine has codimension at least $2$. 

Hence to prove all the assertions of the Theorem, it suffices to show that any maps of spheres of dimension $2j+2,2j+3$ or balls of dimension $2j+3,2j+4$ respectively, can be pushed off the spines of all intersections of fewer of the handlebodies than $k-j$. But this follows immediately by our definition and constructions of multisections. Therefore we can push the spheres and balls into the appropriate intersection components and get epimorphisms or isomorphisms to the appropriate higher homotopy groups as claimed. 
\end{proof}



\section{Constructions and examples}
\label{sec:examples}

We first explain in \S\ref{sec:constructing with symmetric representations} how the symmetric representations of \cite{RT} can be used to construct multisections and give a number of applications of this approach: multisections of spheres and real projective spaces, and the three general constructions of \emph{products}, \emph{generalised multisections} and \emph{twisted multisections}. 
  We also discuss \emph{connected sums} in \S\ref{subsec:conn sums}, a \emph{Dirichlet construction} in \S\ref{sec:Dirichlet}, and the case of \emph{manifolds with non-empty boundary} in \S\ref{sec:multisections for non-closed}.


\subsection{Constructing multisections using symmetric representations}
\label{sec:constructing with symmetric representations}
\label{subsec:applications}

Given a triangulated $n$--manifold $|K|\to M$ with the property that the degree of each $(n-2)$--simplex is even, the authors defined a \emph{symmetric representation} $\pi_1(M)\to \Sym(n+1)$ in \cite{RT} as follows. Pick one $n$--simplex as a base, choose a bijection between its corners and $\{1, \ldots, n+1\}$ and then \emph{reflect} this labelling across its codimension-one faces to the adjacent $n$--simplices. This induced labelling is propagated further and if one returns to the base simplex, one obtains a permutation of the vertex labels. Since the dual 1--skeleton carries the fundamental group, it can be shown that this gives a homomorphism $\pi_1(M)\to \Sym(n+1).$ See \cite[\S2.3]{RT} for the details.
For example, the symmetric representation associated to any barycentric subdivision is trivial, since the labels correspond to the dimension of the simplex containing that vertex in its interior, but there may be more efficient even triangulations with this property. 

The symmetric representation can also be used to propagate partitions of the vertices of the base simplex; this is done in \cite[\S2.5]{RT} for partitions into two sets, but extends to arbitrary partitions. One then obtains a \emph{induced representation}, usually into a symmetric group of larger degree. The aim in \cite{RT} was to obtain information on the topology of a manifold from a non-trivial symmetric representation arising from a triangulation with few vertices. Our needs in this paper are opposite, and we wish to use the symmetric representations to identify triangulations to which we can apply our constructions without barycentric subdivision. So we either want the \emph{orbits} of the vertices under the symmetric representation to give a partition satisfying the conditions in our constructions; or we ask for partitions of the vertices with the property that the induced representation is trivial. 

The main properties to check for a given partition of the vertices are that the graphs spanned by the partition sets are connected, and, in even dimensions, that the dimension of the spine drops when intersecting all but one of the handlebodies.


The following are applications of this approach.

\textbf{Spheres.}
The $n$-sphere can be obtained by doubling an $n$-simplex. This is an even triangulation with trivial symmetric representation. In odd dimensions, this induces a multisection that is a division of a $(2k+1)$--dimensional sphere into $k+1$ balls, each corresponding to a pair of vertices of the $n$-simplex. In even dimensions this is a division of a $(2k)$--dimensional sphere into $k+1$ balls, one corresponding to a single vertex and each of the remaining ones corresponding to a pair of vertices; here one obtains a multisection directly from the partition function without stellar subdivision. Moreover, symmetries of the triangulation permuting the partition sets of the vertices interchange all handlebodies in odd dimensions; in even dimensions they fix the handlebody corresponding to the singleton and act transitively on the remaining ones.

This is the \emph{standard trisection} of a standard sphere, and it also leads to the notion of a \emph{standard trisection of a ball}, by considering the restriction to one of the two $n$-simplices.

\textbf{Projective spaces.}
A symmetric triangulation of the $n$--sphere is obtained by the crosspolytope (see \cite{Zi}) , which is the set of vectors $(x_0,x_1, \dots x_{n})$ in $\mathbb R^{n+1}$ satisfying $\Sigma_i|x_i| = 1$. So there are $2^n$ $n$--simplices. This is invariant under the antipodal map and so descends to an even triangulation of $\mathbb{R}P^n$. We can label the vertices of the triangulation of $S^n$ by the position $i$ of the unique coordinate $x_i$ with $|x_i|=1$ . Pairs of vertices with $x_i = \pm 1$ are interchanged by the antipodal map so this descends to a labelling of the $n$ vertices of the triangulation of $\mathbb{R}P^n$. 

Assume first that $n=2k+1$. Then we can pair the vertices $2i,2i+1$ for $0 \le i \le k$. For $k=1$, the multisection is the Heegaard splitting of genus one, giving two solid tori. For $k>1$, the multisection has $k+1$ handlebodies of genus $1$, whence equivalent with $S^1 \times B^{n-1}.$ In fact there are symmetries of the triangulation interchanging all the handlebodies, so they must all have the same genus. 

Suppose next that $n=2k$. We can then pair up vertices $2i-1,2i$ for $1 \le i \le k$ leaving $0$ as a singleton partition set. It is necessary to perform stellar subdivisions of facets with vertices labelled $i$ with $i>0$. For the vertices labelled $2i-1,2i$ for $1 \le i \le k$, there is still a symmetry interchanging all the handlebodies corresponding to the partition sets. These handlebodies are equivalent with $S^1 \tilde \times B^{n-1},$ and hence have genus $1.$ In contrast, the handlebody corresponding to the partition set labelled $0$ is of genus $2^{2k-1}$. 

\textbf{Generalised multisections.}
Suppose that $M$ is a triangulated $n$--manifold with an even triangulation with trivial symmetric representation. As above, given any triangulation, the first barycentric subdivision has this property. We can define further multisections as follows.

Suppose  that $n=3k+2$. Assume that we have partition sets $P_0,P_1, \ \  \dots \ \  P_{k}$ where the sets have three vertices in every $n$-simplex. We then map each $n$--simplex to the $k$--simplex by mapping each partition set to a vertex of this $k$--simplex. It is then easy to verify that we obtain a division of $M$ into $k+1$ regions, and each region has a $2$--dimensional spine, given by the union of all the $2$--simplices in each $n$--simplex with all vertices in the same partition set. In this case, the manifold $\Sigma$ which is the intersection of all the handlebodies, is closed of dimension $2k+2$. Again we can arrange that the induced cubing of $\Sigma$ is negatively curved and each intersection of a proper subcollection has a spine of low dimension. 

Another interesting example is to have two partition sets of size $k^\prime, k^\star$ of the vertices of each $n$--simplex, so that $k^\prime+ k^\star = n+1$. We assume that both $k^\prime >1, k^\star >1$. The induced decomposition is a bisection into two regions with spines of dimension $k^\prime, k^\star$. Given a handle decomposition of $M$, this is similar to a hypersurface which is the boundary of the region containing all the $i$--handles for $0 \le i \le k^\prime$. 

Finally a very specific example is a $6$--manifold $M$ with three partition sets of respective sizes $2,2,3$. This induces a trisection of $M$ into three regions, where two are handlebodies and the third has a 2--dimensional spine. 

\textbf{Twisted multisections.}
Suppose a closed PL $n$--manifold has an even triangulation with a non-trivial symmetric representation. Assume also that the symmetry preserves our standard partition of the vertices, i.e.\thinspace every symmetry mapping produces a permutation of the partition sets of vertices. Then there is an associated `twisted' multisection, which we illustrate with a simple example---the general construction then becomes clear. 

Assume $M$ is a 5--manifold that admits an even triangulation with a symmetric representation with image $\Z_3$. Also assume this symmetry is a permutation of the form $(012)(345)$ of the labelling of the vertices. In this case, we choose as partition sets $\{0,3\},\{1,4\},\{2,5\}.$ Then these are permuted under the action of the symmetric mapping. The edges joining these three pairs of vertex sets clearly form a connected graph $\Gamma$. 

A regular neighbourhood of $\Gamma$ then forms a single handlebody $H$ whose boundary is glued to itself to form $M$. The handlebody $H$ lifts to three handlebodies in a regular 3--fold covering space $\widetilde{M}$ of $M$ and these give a standard trisection of $\widetilde{M}$. The covering transformation group $\Z_3$ permutes the handlebodies and preserves the central submanifold. If the initial triangulation is flag, then the lifted triangulation is flag and hence the central submanifold has a non-positively curved cubing on which the covering transformation group acts isometrically. Hence the quotient, which embeds in $\partial H,$ also has a non-positively curved cubing. 

\subsection{Connected sums}
\label{subsec:conn sums}
Suppose the $n$--manifolds $M$ and $M^\prime$ have multisections. Then we can define the connected sum of these by removing small open $n$--disks $B,$ $B^\prime$ from $M,$ $M^\prime$ respectively, so that the closures meet the multisections in the standard multisection of the ball. We can then glue the multisections of $M \setminus B$ and $M^\prime \setminus B^\prime$ with a gluing map that matches the multisections along the boundary spheres to form the connected sum multisection, and the resulting connected sum depends on the way the handlebodies of $M$ and $M^\prime$ are matched up.


\subsection{Dirichlet Constructions}
\label{sec:Dirichlet}

We give two \emph{Dirichlet interpretations} of multisections---one works in general using a polyhedral metric, and the other works in special cases for symmetric spaces.

The general construction uses the polyhedral metric defined by each simplex having the standard Euclidean metric. Such metrics arise in Regge calculus (see \thinspace \cite{Regge, WT1992, RW2000}), which applies PL topology to general relativity and quantum gravity. Suppose $M$ is given with a triangulation and partition function supporting a multisection.

Assume that $\Gamma_i$ is the graph of all edges corresponding to the $i$\textsuperscript{th} partition set of the vertices. Then the handlebodies of a multisection are given by 
$$H_i = \{x \in M \mid d(x,\Gamma_i) \le d(x,\Gamma_j) \;\forall j: j \ne i\}.$$
In other words, the multisection is defined by a Dirichlet construction, where instead of considering distance from a set of distinct points, we use distance from a collection of disjoint graphs. In particular, each handlebody is the set of all points, which are not further from one of the partition edge graphs than any of the others.

Moreover the multisection and central submanifolds, i.e.\thinspace the submanifolds arising from intersections, are then given by 
$$H_{i_1} \cap \dots \cap H_{i_s} = \{x\in M \mid d(x,\Gamma_{i_1})= \dots = d(x,\Gamma_{i_s})\}.$$
In particular, in some symmetric spaces we can use this Dirichlet viewpoint to achieve that all multisection submanifolds are totally geodesic relative to the appropriate symmetric space metric.  For these submanifolds are intersections of piecewise totally geodesic hypersurfaces, given by sets of points equidistant from two distinct points. This works in the case that the handlebodies are actually cells, so our Dirichlet construction is the usual one. 

As an example, first consider $S^n$ with the standard round metric.  First assume $n=2k+1$. We can pick equally spaced points $x_1,x_2, \dots x_{k+1}$ from a totally geodesic embedded $S^k$. These points form an orbit of the isometric action of the symmetric group $\Sym(k+1)$ on $S^n$, where the action preserves the standard Hopf link $S^k \cup S^k$ and the points lie in one of these copies of $S^k$. Then the standard Dirichlet construction gives the standard multisection for $S^n$.

If $n=2k$, the same construction works by instead considering a totally geodesic $S^{n-1}$ in $S^n$ and choosing $k$ points symmetrically situated in $S^{n-1}$ to perform the Dirichlet construction. Note in both even and odd cases, all the components of the multisections are totally geodesic submanifolds.

A similar construction applies to $\mathbb{C}P(n)$, so the multisection submanifolds are also piecewise totally geodesic in the standard symmetric space metric. In this case, the action of the symmetric group $S_{n+1}$ is given by permuting coordinates, so we can take our orbit of points as 
$$[1,0, \dots 0], [0,1,0, \dots, 0], \dots, [0,\dots,0, 1].$$ 

Here is a simple example of a multisection Dirichlet construction using graphs rather than points. Start with the round metric on $S^5$ considered as the unit sphere in $\mathbb C^3$. Choose three geodesic circles given by $z\mapsto(z,0,0),$ $z\mapsto(0,z,0),$ $z\mapsto(0,0,z)$ for $|z|=1$. The Dirichlet regions relative to these circles are given by 
$$\{\;(z_1,z_2,z_3)\;\mid\; |z_i| \le |z_j|,|z_k|\;\}$$
for $\{i,j,k\} = \{1,2,3\}$. The central submanifold is the flat 3-torus defined by $|z_1|=|z_2|=|z_3|$. Each handlebody is $S^1 \times B^4$, and each intersection of pairs of handlebodies is a copy of $S^1 \times S^1 \times B^2$.

This multisection is in fact the pull back of the multisection of $\mathbb{C}P(2)$, viewing $S^5$ as a circle bundle over $\mathbb{C}P(2)$. The same construction works for all odd dimensional spheres---they have multisections where all the handlebodies are a product of a circle and a ball.

\subsection{Products}
A natural product of multisections can be defined via the categorical product or the external join. However, both of these operations generally do not result in manifolds, and it is an interesting topic for investigation to relate them to a multisection of the direct product in a canonical way. 

We will briefly describe one possible approach for these product operations. Our starting point are PL manifolds $N^n$ and $M^m$ with triangulations supporting multisections. 

The \emph{categorical product} of $N$ and $M$ is a simplicial complex of dimension $(n+1)(m+1)-1,$ which is homotopy equivalent to the direct product $N\times M$ \cite[Prop. 15.23]{Ko}. The set of vertices of a simplex in the categorical product is the direct product of the vertices of a simplex in $N$ with the set of vertices of a simplex in $M.$ Hence a partition of the vertices can be canonically defined from the partitions for $N$ and $M:$ $(v,w)$ and $(a,b)$ are in the same partition set if and only if
either $v=a$ and $\{w, b\}$ is a partition set of $M$ or $\{v, a\}$ is a partition set of $N$ and $w=b.$

The \emph{join} of $N$ and $M,$ denoted $N\star M,$ is a simplicial complex of dimension $n+m+1,$ which is equivalent to $N \times M \times [0,1]$ with $N \times M \times \{0\}$ collapsed to $N$ and $N \times M \times \{1\}$ collapsed to $M.$ In this case, $N \times M \times \{\frac{1}{2}\}$ is equivalent with $N\times M,$ and has a natural cell decomposition. In the case $N$ and $M$ are standard spheres, $N\star M$ is again a standard sphere \cite[Prop. 2.23]{RS}. Each top-dimensional simplex $\sigma$ of $N\star M$ is of the form $\sigma^n \star \sigma^m$ with set of vertices the union of the vertices of $\sigma^n$ and $\sigma^m.$ A partition of the vertices of $\sigma$ can be defined by simply taking the partition sets of $\sigma^n$ and $\sigma^m,$ unless both $n$ and $m$ are even, in which case the union of the two singletons is a partition set. 


\subsection{Manifolds with non-empty boundary}
\label{sec:multisections for non-closed}

To deal with compact manifolds with non-empty boundary, there are two different models in dimension $3,$ which one can take as a launching point. One gives a decomposition of a $3$--manifold into two compression bodies, which are obtained by attaching $1$--handles to products of closed surfaces and intervals along one of the boundary surfaces in each such product. The other decomposes the manifold into two handlebodies, which are glued along subsurfaces of their boundaries. 
Interesting examples of the former can be obtained as follows.
Birman~\cite{Bi} shows that each closed orientable $3$--manifold $M$ bounds a $4$--manifold $W$ given by gluing two $4$--balls together along a Heegaard handlebody in their boundary $3$--spheres. Hence this is a trisection of $W$ with one component a collar of the boundary $\partial W = M$ and the other two components $4$--balls. A general theory, analogous to the one in this paper, can be developed by working with the interior of the manifold and using triangulations with ideal vertices.

Gay and Kirby~\cite{GK} give a different construction for multisections of 4--manifolds with boundary, motivated by their study of Morse 2--functions. Here the boundary of the 4--manifold has an induced fibration or open book decomposition. It is an interesting topic for further investigation to formulate a discrete version of this approach to all dimensions.



\address{School of Mathematics and Statistics, The University of Melbourne, VIC 3010, Australia} 
\address{School of Mathematics and Statistics, The University of Sydney, NSW 2006, Australia} 
\email{rubin@ms.unimelb.edu.au} 
\email{tillmann@maths.usyd.edu.au} 

\Addresses


\end{document}